%% file: main.tex
\RequirePackage[l2tabu, orthodox]{nag}
\documentclass[a4paper,12pt]{amsart}
\usepackage[all, warning]{onlyamsmath}

\usepackage[margin=1.2in,marginparwidth=2cm]{geometry}

\usepackage{mathtools,amssymb}
\usepackage{amsthm}
\usepackage{mleftright}

\usepackage[T1]{fontenc}
\usepackage[sc,noBBpl]{mathpazo}
\usepackage[scaled]{helvet}
\usepackage{bm}
\usepackage{eucal}

\usepackage{graphicx}
\usepackage{tikz-cd}
\usepackage{xcolor}

\usepackage{prettyref}
\PassOptionsToPackage{hyphens}{url}
\usepackage[hidelinks,pdfusetitle,bookmarksnumbered,final]{hyperref}

\usepackage{mainstyle}

\newenvironment{DIFnomarkup}{}{}

\begin{document}

\title[On the effective generation of direct images in mixed characteristic]{On the effective generation of direct images of pluricanonical bundles in mixed characteristic}
\author{Hirotaka Onuki}
\date{\today}

\begin{abstract}
We present an effective global generation result for direct images of pluricanonical bundles in mixed characteristic. This is a mixed characteristic analog of Ejiri's theorem \cite{Ejir24} in positive characteristic and the theorem of Popa and Schnell \cite{PoSc14} regarding their Fujita-type conjecture in characteristic zero. As an application, we establish a weak positivity statement for the relative canonical sheaf of a smooth morphism in mixed characteristic.
\end{abstract}

\maketitle

\tableofcontents

\input{content.tex}

\bibliographystyle{amsalpha}
\bibliography{myref}

\end{document}

%% file: content.tex
\section{Introduction}
A special case of Fujita's freeness conjecture states that for a globally generated ample line bundle $\sL$ on an $n$-dimensional nonsingular projective variety $X$ over $\CC$, the adjoint bundle
\[
    \sOm_X \otimes \sL^l
\]
is globally generated for all $l \ge n+1$.
While Fujita's freeness conjecture itself is still unsolved, this special case follows from an application of the Castelnuovo-Mumford regularity and the Kodaira vanishing theorem.
An analogous result was shown in positive characteristic by using Serre vanishing and the Frobenius morphism instead of Kodaira vanishing (see \cite{Smit97, HarN:2003, Keer08}).
Very recently, a mixed characteristic analog was established in \cite{BMP+23, HLS21}, which relies on the recent development of commutative algebra in mixed characteristic.

Popa and Schnell \cite{PoSc14} generalized the above-stated case of Fujita's freeness conjecture to the relative setting: for a surjection $f \colon X \to Y$ between nonsingular projective varieties over $\CC$ with $\dim Y = n$ and a globally generated ample line bundle $\sL$ on $Y$, the direct image
\[
    f_*\sOm_X^m \otimes \sL^l
\]
of the pluricanonical bundle is globally generated for all $m \ge 1$ and $l \ge m(n + 1)$.
More generally, they proved the global generation of
\[
    f_*\sO_X(m(\dK_X + \Delta)) \otimes \sL^l
\]
for a log canonical pair $(X, \Delta)$.

On the other hand, an analogous statement in positive characteristic has a counterexample with $m = 1$: for the fibration $g \colon X \to \scPP^1$ of Moret-Bailly \cite{More81} over an algebraically closed field of characteristic $p > 0$, the sheaf
\[
    g_* \sOm_X \otimes \sO_{\scPP^1}(2) \isom \sO_{\scPP^1}(-1) \oplus \sO_{\scPP^1}(p)
\]
is not globally generated (see also \cite[Theorem 3.16]{ShZh20}).
Therefore, we need to impose some conditions to obtain a global generation result for direct images of pluricanonical bundles in positive characteristic. Ejiri \cite{Ejir24, Ejir23b} established such a result with the assumption that $m$ is sufficiently large and the canonical divisor of $X$ is $f$-ample:
\begin{theorem}[{\cite[Theorem 6.11 (1)]{Ejir24}, \cite[Theorem 3.3]{Ejir23b}}]
Let $k$ be an algebraically closed field of characteristic $p > 0$ and let $f \colon X \to Y$ be a surjection between projective varieties over $k$ with $X$ normal.
Let $\Delta$ be an effective $\QQ$-divisor on $X$ such that $\dK_X + \Delta$ is $\QQ$-Cartier with index $r \ge 1$ not divisible by $p$.
Suppose that $V \subset Y$ is an open set and $U \coloneqq f^\inv(V)$ is the preimage of $V$.
Assume in addition that the following two conditions hold.
\begin{itemize}
\item $(\dK_X + \Delta) \restrict{U}$ is $f \restrict{U}$-ample.
\item $(U, \Delta \restrict{U})$ is $\Fr$-pure. This assumption is satisfied, for example, if $U$ is regular, $\Supp \Delta \restrict{U}$ is SNC, and all the coefficients of $\Delta \restrict{U}$ are at most $1$.
\end{itemize}
Let $\sL$ be a globally generated ample line bundle on $Y$, and set $n \coloneqq \dim Y$.
Then
\[
    f_* \sO_X(m(\dK_X + \Delta)) \otimes \sL^l
\]
is globally generated over $V$ for all sufficiently large multiples $m$ of $r$ and for all $l \ge m(n+1)$ (see \myref{def: gg} for the definition of global generation over an open set).
\end{theorem}

Our main theorem is the following mixed characteristic analog of Ejiri's theorem.
\begin{main}[see \myref{thm: RelFuj} and \myref{rmk: RelFujSing}]\label{main: main}
Let $(R, \maxmR)$ be a complete DVR of mixed characteristic $(0, p)$ and let $f \colon X \to Y$ be a surjection between integral flat projective $R$-schemes with $X$ normal.
Let $\Delta$ be an effective $\QQ$-divisor on $X$ such that $\dK_X + \Delta$ is $\QQ$-Cartier with index $r \ge 1$.
Suppose that $V \subset Y$ is an open set and $U \coloneqq f^\inv(V)$ is the preimage of $V$.
Assume in addition that the following two conditions hold.
\begin{itemize}
\item $(\dK_X + \Delta) \restrict{U}$ is $f \restrict{U}$-ample.
\item For every $x \in U$, $(\completion{\sO_{X,x}}, \Delta \restrict{\completion{\sO_{X,x}}})$ is $\text{BCM}$-regular (see \myref{def: BCMReg} for the definition) if the residue characteristic of $\sO_{X,x}$ is $p$, and is KLT if it is $0$.
This assumption is satisfied, for example, if $U$ is regular, $\floor{\Delta \restrict{U}} = 0$ and $\Supp \Delta \restrict{U}$ is SNC.
\end{itemize}
Let $\sL$ be a globally generated ample line bundle on $Y$.
Set $n$ to be the dimension of the special fiber $Y_{\maxmR}$ of $Y$.
Then
\[
    f_* \sO_X(m(\dK_X + \Delta)) \otimes \sL^l
\]
is globally generated over $V$ for all sufficiently large multiples $m$ of $r$ and for all $l \ge m(n+1)$.
\end{main}

Here, $R$ is assumed to be a DVR only to simplify the conditions on the singularities of $X$. A similar theorem holds when $R$ is of higher dimension (see \myref{thm: RelFuj} for details).

We also treat the case $m = 1$. We show a global generation result for $f_* \sOm_X \otimes \sL^l$ under mild conditions on the singularities when $R$ is a DVR (see \myref{rmk: gg-alt-dvr} and \myref{cor: gg-alt-tau}).
When $R$ has higher dimension, we obtain a similar result with more restrictive assumptions on the singularities (see \myref{thm: gg-alt}).

\subsection{Applications}
Popa and Schnell \cite{PoSc14} utilized their theorem to reprove Viehweg's weak positivity theorem \cite{Vieh83}: if $f \colon X \to Y$ is a surjection between nonsingular projective varieties over $\CC$, then
\[
    f_*\sOm_{X/Y}^m
\]
is weakly positive for every $m \ge 1$.
Building upon \cite{PoSc14}, Dutta and Murayama \cite{DuMu19} obtained a weak positivity theorem for log canonical pairs.
Patakfalvi \cite{Pat14} and Ejiri (\cite{Ejir17}, \cite[Theorem 6.11 (2)]{Ejir24}) proved a similar result in positive characteristic.

We obtain the following weak positivity result in mixed characteristic by using \myref{main: main}.
\begin{main}[\myref{thm: Wp}]\label{main: wp}
Under the assumption of \myref{main: main}, suppose further that $X$ and $Y$ are regular, $f$ is smooth, $\floor{\Delta \restrict{U}} = 0$, and $\Supp \Delta\restrict{U}$ is relatively SNC over $V$ (see \myref{def: rel-snc} for the definition).
Then $\dK_{X/Y} + \Delta$ is weakly positive over $U$ (see \myref{def: Wp} for the definition).
\end{main}
The smoothness conditions in \myref{main: wp} are imposed to ensure that the fiber products of copies of $X$ over $Y$ also have only mild singularities.

\myref{main: wp} is enough for many applications, although it is weaker than the weak positivity of $f_* \sO_X(m(\dK_{X/Y} + \Delta))$.

As a consequence of \myref{main: wp}, we can show that under the assumption of \myref{main: wp}, the direct image
\[
    f_*\sO_X(m(\dK_{X/Y} + \Delta)) \otimes \sOm_Y \otimes \sL^{n+1}
\]
is globally generated over $V$ (see \myref{prop: Fujino}).
This is a mixed characteristic analog of Fujino's theorem \cite[Theorem 1.7]{Fujn23} in characteristic zero and Ejiri's theorem \cite[Theorem 1.9]{Ejir24} in positive characteristic.

Using \myref{main: wp}, we prove a result on the smooth descent of the positivity of the anticanonical divisor in mixed characteristic, although it can also be shown without \myref{main: wp} (see \myref{prop: Fano} for details).
The equal characteristic case was proven by \cite[Corollary 2.9]{KMM92} (see also \cite[Corollary 3.15]{Debarre}, \cite{FuGo14} and \cite{Ejir19}).

\subsection{Outline of the proof}
We describe the key lemma for the proof of \myref{main: main}.

Hacon, Lamarche, and Schwede \cite{HLS21} defined the $\bmplus$-test ideal
\[
    \sHLS(\sO_X, \Delta) \subset \sO_X\text,
\]
which is a mixed characteristic analog of the test ideal $\tau(\sO_X, \Delta)$ in positive characteristic.
They obtained a useful global generation theorem involving the $\bmplus$-test ideal \cite[Theorem B]{HLS21}.
We give a natural generalization of the $\bmplus$-test ideal to the relative setting: for a line bundle $\sM$ on $X$, we define a subsheaf
\[
    \relBOM{f}{X}{\Delta}{\sM} \subset f_*\mleft(\sHLS\mleft(\sO_X, \Delta\mright) \otimes \sM\mright)\text,
\]
which is a mixed characteristic analog of the Frobenius stable direct image $S^0f_*(\sigma(X, \Delta) \otimes \sM)$ in \cite{HaXu15}.

The key to the proof of \myref{main: main} is the following global generation result for $\relBOM{f}{X}{\Delta}{\sM}$, which generalizes \cite[Theorem B]{HLS21}.
\begin{main}[\myref{thm: RelGgOp}]\label{main: relB0}
Suppose that $f$, $X$, $Y$, $\sL$, and $\Delta$ are as in \myref{main: main}.
Let $\sM = \sO_X(M)$ be a line bundle on $X$.
Suppose that $V' \subset Y$ is an open set, $U' \coloneqq f^\inv(V')$ is the preimage of $V'$, and $M - \dK_X - \Delta$ is semiample over $U'$ and big (see \myref{def: Wp} for the definition of semiampleness over an open set).
Then
\[
    \relBOM{f}{X}{\Delta}{\sM} \otimes \sL^n
\]
is globally generated over $V'$.
\end{main}
\myref{main: relB0} is regarded as a mixed characteristic analog of \cite[Theorem 1.2]{Ejir23a}, which is a global generation theorem for the Frobenius stable direct image.

\begin{acknowledgement}
The author wishes to express his sincere gratitude to his advisor Shunsuke Takagi for suggesting the problems and giving much helpful advice. The author is indebted to Shou Yoshikawa for many insightful suggestions, and to Sho Ejiri for numerous valuable suggestions and for generously sharing his ideas.
The author would like to thank the anonymous referee for a thorough and careful reading of the manuscript.
\end{acknowledgement}

\section{Preliminaries}
Throughout this article, $(R, \maxmR)$ denotes a complete Noetherian local domain of residue characteristic $p > 0$.

\subsection{Duality}
Let $E \coloneqq E_R(R / \maxmR)$ be the injective hull of the residue field. The Matlis duality functor $(\blank) ^ \dual \coloneqq \Hom_R(\blank, E)$ gives the anti-equivalence between the Noetherian $R$-modules and the Artinian $R$-modules \citestacks{08Z9}. Fix a normalized dualizing complex $\cpOm_R$ of $R$. For a quasi-projective $R$-scheme $\pi \colon X \to \Spec R$, define $\cpOm_X \coloneqq \pi^! \cpOm_R$ and $\sOm_X \coloneqq H^{-{\dim X}}(\cpOm_X)$.

For an $R$-scheme $X$ and an $\sO_X$-module $\sF$, write $\fHRRGa{i}{\maxmR}(X, \sF)$ for $H^i \fRR\Gamma_{\maxmR} \fRR\Gamma(X, \sF)$.

\subsection{Absolute integral closures}
Consider a normal integral projective $R$-scheme $X$. Fix an algebraic closure $\algcl{\rK(X)}$ of the function field $\rK(X)$ of $X$. The \emph{absolute integral closure} of $X$, denoted by $\nu \colon X^\abscl \to X$, is the normalization of $X$ in $\algcl{\rK(X)}$. By \emph{limit over all finite covers of $X$}, we mean the limit over all finite surjections $g \colon X' \toepi X$ from a normal integral scheme $X'$ equipped with an inclusion $\rK\mleft(X'\mright) \subset \algcl{\rK(X)}$.

In this paper, we often consider a surjection $f \colon X \to Y$ to an integral projective $R$-scheme $Y$.
In typical situations, we can assume $R = \mH^0(Y, \sO_Y)$ as follows. Set $R' \coloneqq \mH^0(Y, \sO_Y)$. Then $Y \to \Spec R' \to \Spec R$ is the Stein factorization of $Y \to \Spec R$.
Since $R'$ is finite over a complete Noetherian local domain $R$, it is a finite product of complete Noetherian local domains by \citestacks{04GH} and \citestacks{04GM}.
Since $Y$ is integral, it follows that $R'$ itself is a complete Noetherian local domain.
Hence, replacing $R$ by $R'$, we can assume $R = \mH^0(Y, \sO_Y)$.

We quote the following important result.
\begin{theorem}[Bhatt's vanishing {\cite[Corollary 3.7]{BMP+23}}, \cite{Bhat20}]\label{thm: Bhatt}
Let $X$ be a normal integral projective $R$-scheme and let $L$ be a big and semiample $\QQ$-Cartier divisor on $X$. Then
\[
    \fHRRGa{i}{\maxmR}\mleft(X^\abscl, \sO_{X^\abscl}(-\nu^*L)\mright) = 0
\]
for all $i < \dim X$.
\end{theorem}

\begin{proof}
If $L$ is Cartier, then the assertion is \cite[Corollary 3.7]{BMP+23}.
The general case is reduced to this case by replacing $X$ with a sufficiently large finite cover $X'$, since $L$ is pulled back to a Cartier divisor on $X'$ (see \cite[Remark 4.4]{BMP+23}).
\end{proof}

\subsection{Base loci}
\begin{definition}\label{def: gg}
Let $X$ be a projective $R$-scheme and let $U$ be an open set in $X$. A coherent sheaf $\sF$ on $X$ is said to be \emph{globally generated over $U$} if the natural map $\mH^0(X, \sF) \otimes_R \sO_X \to \sF$ is surjective on $U$.
\end{definition}

\begin{definition}
Let $X$ be a normal integral projective $R$-scheme and let $D$ be a $\QQ$-Cartier divisor on $X$ with Cartier index $r$. The \emph{stable base locus} $\BB(D)$ of $D$ is defined to be $\bigcap_{m \ge 1} \scRed{\Bs(mrD)}$.
The \emph{augmented base locus} and the \emph{restricted base locus} of $D$ are defined by $\BBp(D) \coloneqq \bigcap_{A} \BB(D - A)$ and $\BBm(D) \coloneqq \bigcup_{A} \BB(D + A)$ respectively, where the intersection and the union are taken over all ample $\QQ$-divisors $A$ on $X$.
\end{definition}

\begin{definition}[{\cite[\S 2.5]{BMP+23}}]\label{def: nef-etc}
Let $X$ be a normal integral projective $R$-scheme and let $D$ be a $\QQ$-Cartier divisor on $X$. Write $f \colon X \to \Spec R$ for the structure morphism.
We say that $D$ is \emph{big} if $D \restrict{X_\eta}$ is big, where $\eta$ is the generic point of $f(X)$: equivalently, $\BBp(D) \neq X$.
We say that $D$ is \emph{semiample} if $mD$ is a globally generated Cartier divisor for some $m > 0$: equivalently, $\BB(D) = \emptyset$.
We say that $D$ is \emph{pseudoeffective} if $D + A$ is big for every ample $\QQ$-Cartier divisor $A$: equivalently, $\BBm(D) \neq X$.
We say that $D$ is \emph{nef} if $D + A$ is ample for every ample $\QQ$-Cartier divisor $A$: equivalently, $\BBm(D) = \emptyset$.
Note that $D$ is nef if and only if $D \cdot C \ge 0$ for every proper curve $C$ in the closed fiber $X_{\maxmR}$.
\end{definition}

Note that $D$ is ample if and only if $\BBp(D) = \emptyset$.

\begin{definition}\label{def: Wp}
Let $X$ be a normal integral projective $R$-scheme and let $D$ be a $\QQ$-Cartier divisor on $X$. Let $U$ be an open subset of $X$. We say that $D$ is \emph{weakly positive over $U$} (resp.~\emph{semiample over $U$}) if $\BBm(D) \cap U = \emptyset$ (resp.~$\BB(D) \cap U = \emptyset$).
\end{definition}

We say that $D$ is \emph{weakly positive} (without specifying an open subset of $X$) if $D$ is weakly positive over $U$ for some nonempty open subset $U$ of $X$ (cf.~\cite[Definition 4.12]{Ejir24}, \cite[Theorem 1.1]{BKKMSU15}). We do not use this terminology in this paper.

\begin{remark}\label{rmk: nef-cl-fib}
Let $X$ be a normal integral projective $R$-scheme and let $\sL$ be a line bundle on $X$.
If the line bundle $\sL_{\maxmR}$ on the closed fiber $X_{\maxmR}$ is nef, then $\sL$ is nef.

This assertion follows from \cite[Chapitre III, Th\'eor\`eme 4.7.1]{EGAIII} (cf.~\cite[Theorem 1.2.17]{LazI}).
Alternatively, we obtain the assertion from the characterization of nefness in terms of intersection numbers, noting that $C$ is contained in $X_{\maxmR}$, with notation as in \myref{def: nef-etc}.
\end{remark}

\section{\texorpdfstring{$\bmplus$}{+}-test ideals and \texorpdfstring{$\bmplus$}{+}-stable direct images}
The aim of this section is to extend the definition of the $\bmplus$-test ideal by Hacon, Lamarche, and Schwede \cite{HLS21} to the relative setting.
This extension, which we call the $\bmplus$-stable direct image, is used in the proof of \myref{main: main}.

\subsection{\texorpdfstring{$\bmplus$}{+}-stable sections and \texorpdfstring{$\bmplus$}{+}-test ideals}
We review the definitions of the space of the $\bmplus$-stable sections \cite{BMP+23} (cf.\ \cite{TaYo23}) and the $\bmplus$-test ideal of \cite{HLS21}.

\begin{remark}\label{rmk: fin-trace}
Let $X$ be a normal integral projective $R$-scheme and let $M$ be a (Weil) divisor on $X$, not necessarily $\QQ$-Cartier.
Let $g \colon X' \to X$ be a finite surjection from a normal integral scheme $X'$.
Then $g^* M$ is defined by the closure of $(g\restrict{U})^*(M\restrict{U})$ in $X'$ where $U = \scReg{X}$ is the regular locus of $X$ (cf.~\cite[Proof of Proposition 5.20]{KoMo98}).

We have the Grothendieck trace map $g_*\sO_{X'}(\dK_{X'}) \to \sO_X(\dK_X)$. Since $M$ is Cartier on $U = \scReg{X}$, we obtain the map $g_*\sO_{X'}(\dK_{X'} + g^*M)\restrict{U} \to \sO_X(\dK_X + M)\restrict{U}$.
Taking the reflexive hulls, we have the map $g_*\sO_{X'}(\dK_{X'}+g^*M) \to \sO_X(\dK_X+M)$. We obtain the composition
\[
g_*\sO_{X'}\mleft(\dK_{X'}+\ceil{g^*(M-B)}\mright)
\tomono g_*\sO_{X'}\mleft(\dK_{X'}+g^*M\mright)
\to \sO_X(\dK_X+M)
\]
for each $\QQ$-divisor $B \ge 0$ on $X$.
\end{remark}

\begin{definition}[{\cite[Section 4]{BMP+23}}]\label{def: mBMP}
Let $X$ be a normal integral projective $R$-scheme, let $M$ be a divisor on $X$ and let $B \ge 0$ be a $\QQ$-divisor on $X$.
\begin{itemize}
\item Define the submodule $\mBMP(X,B,\sO_X(\dK_X+M)) \subset \mH^0(X,\sO_X(\dK_X+M))$ to be
    \[
        \bigcap_{g \colon X' \to X} \im\mleft(\mH^0\mleft(X',\sO_{X'}\mleft(\dK_{X'}+\ceil{g^*(M-B)}\mright)\mright) \to \mH^0\mleft(X,\sO_X(\dK_X+M)\mright)\mright)\text,
    \]
    where $g \colon X' \to X$ runs over all finite surjections from normal integral schemes. The sections in $\mBMP\mleft(X,B,\sO_X(\dK_X+M)\mright)$ are called the \emph{$\bmplus$-stable sections of $\mH^0(X,\sO_X(\dK_X+M))$ with respect to $(X, B)$}.
    We write $\mBMP\mleft(X,\sO_X(\dK_X+M)\mright)$ for $\mBMP\mleft(X,0,\sO_X(\dK_X+M)\mright)$.
\item Let $D = \sum_i D_i$ be a reduced divisor on $X$ with $D_i$ the irreducible components. Fix an integral closed subscheme $D_i^\abscl \subset X^\abscl$ lying over $D_i$, or equivalently, fix an integral closed subscheme $D_{i, X'}$ for each finite cover $g \colon X' \to X$ such that for every finite cover $h \colon X'' \to X'$, $D_{i, X''}$ lies over $D_{i, X'}$. Define the submodule $\mBMPAdj{D}(X,D+B,\sO_X(\dK_X+M)) \subset \mH^0(X,\sO_X(\dK_X+M))$ to be
    \begin{gather*}
        \bigcap_{g \colon X' \to X} \im\mleft(\bigoplus_i \mH^0\mleft(X',\sO_{X'}\mleft(\dK_{X'}+D_{i,X'}+\ceil{g^*(M-B-D)}\mright)\mright) \mright.\\
        \mleft. \to \mH^0\mleft(X,\sO_X\mleft(\dK_X+M\mright)\mright)\mright)\text,
    \end{gather*}
    where $g \colon X' \to X$ runs over all finite surjections from normal integral schemes.
    The map above comes from the Grothendieck trace map as follows. Set $D_{X'} \coloneqq \sum_i D_{i,X'}$. Then we have $D_{X'} \le g^*D$, and obtain the composition
    \begin{align*}
        &\phantom{\to} g_*\bigoplus_i \sO_{X'}\mleft(\dK_{X'}+D_{i,X'}+\ceil{g^*(M-B-D)}\mright) \\
        &\to g_*\sO_{X'}\mleft(\dK_{X'}+D_{X'}+\ceil{g^*(M-B-D)}\mright) \\
        &\tomono g_*\sO_{X'}\mleft(\dK_{X'}+\ceil{g^*(M-B)}\mright) \\
        &\to \sO_X\mleft(\dK_X+M\mright)\text{,}
    \end{align*}
    where the first map is induced by the diagonal map.
    The definition of $\mBMPAdj{D}(X,D+B,\sO_X(\dK_X+M))$ is independent of the choice of $D_i^\abscl \subset X^\abscl$ by \cite[Lemma 4.23]{BMP+23}.
\end{itemize}
\end{definition}

\begin{remark}\label{rmk: B0alt}
If $M - B$ is $\QQ$-Cartier, we can define $\mBMP(X,B,\sO_X(\dK_X+M))$ in terms of alterations by \cite[Corollary 4.13, Definition 4.2]{BMP+23}:
we have
\begin{align*}
    &\mBMP(X,B,\sO_X(\dK_X+M)) =\\
    &\bigcap_{g \colon X' \to X} \im\mleft(\mH^0\mleft(X',\sO_{X'}\mleft(\dK_{X'}+\ceil{g^*(M-B)}\mright)\mright) \to \mH^0\mleft(X,\sO_X(\dK_X+M)\mright)\mright)\text,
\end{align*}
where $g \colon X' \to X$ runs over all projective generically finite morphisms from normal integral schemes. Here, the map in the intersection is induced by the Stein factorization of $g$ and \myref{rmk: fin-trace}.
\end{remark}

\begin{definition}[{\cite[Definition 4.3, 4.14, 4.4]{HLS21}}]\label{def: sHLS}
Let $X$ be a normal integral projective $R$-scheme.
\begin{itemize}
    \item Let $B \ge 0$ be a $\QQ$-divisor on $X$. Suppose that $\sL$ is a very ample line bundle on $X$. Let $\sJ_i \subset \sOm_X \otimes \sL^i$ be the subsheaf generated by $\mBMP\mleft(X, B, \sOm_X \otimes \sL^i\mright)$ for $i > 0$. The \emph{$\bmplus$-test submodule} of $(X,B)$, denoted by $\sHLS(\sOm_X, B)$, is defined by $\sHLS(\sOm_X, B) \coloneqq \sJ_i \otimes \sL^{-i} \subset \sOm_X$ for $i \gg 0$.
    This definition is independent of the choice of $\sL$.
    We write $\sHLS(\sOm_X)$ for $\sHLS(\sOm_X, 0)$.
    \item Let $B$ be a $\QQ$-divisor on $X$, not necessarily effective. Define
    \[
        \sHLS(\sOm_X, B) \coloneqq \sHLS(\sOm_X, B+H) \otimes \sO_X(H) \subset \sOm_X \otimes \rK(X) \text,
    \]
    where $H$ is a Cartier divisor on $X$ such that $B+H \ge 0$. This subsheaf is independent of the choice of $H$ by \cite[Lemma 4.8 (b)]{HLS21}.
    \item Let $\Delta \ge 0$ be a $\QQ$-divisor on $X$. Define
    \[
        \sHLS(\sO_X, \Delta) \coloneqq \sHLS(\sOm_X, \dK_X+\Delta)\text,
    \]
    which is called the \emph{$\bmplus$-test ideal} of $(X,\Delta)$. By~\cite[Lemma 4.18, 4.17]{HLS21}, this is an ideal sheaf of $X$ that is independent of the choice of a canonical divisor $\dK_X$.
    \item Let $B \ge 0$ be a $\QQ$-divisor on $X$ and let $D$ be a reduced divisor on $X$. Suppose that $\sL$ is a very ample line bundle on $X$ and $\sN_i \subset \sOm_X(D) \otimes \sL^i$ is the subsheaf generated by $\mBMPAdj{D}\mleft(X, D+B, \sOm_X(D) \otimes \sL^i\mright)$ for $i > 0$. Define $\sHLSAdjW{X}{D}{B} \coloneqq \sN_i \otimes \sL^{-i} \subset \sOm_X(D)$ for $i \gg 0$.
    This definition is independent of the choice of $\sL$.
\end{itemize}
\end{definition}

\begin{definition}
Let $X$ be a normal integral projective $R$-scheme, let $M$ be a divisor on $X$ and let $B$ be a $\QQ$-divisor on $X$, not necessarily effective. Define
\begin{align*}
    \mBMP(X,B,\sO_X(\dK_X+M)) &\coloneqq \mBMP(X,B+H,\sO_X(\dK_X+M+H)) \\
    &\subset \mH^0(X, \sO_X(\dK_X+M) \otimes \rK(X))\text,
\end{align*}
where $H$ is a Cartier divisor on $X$ such that $B+H \ge 0$.
This definition is independent of the choice of $H$ by \cite[Lemma 3.3 (c)]{HLS21}.
\end{definition}

\begin{remark}[{\cite[Remark 4.3]{BMP+23}}]
Suppose that $X$ is not integral, but normal and equidimensional. Then we can still define the $\bmplus$-stable sections and the $\bmplus$-test ideal by decomposing $X$ into the connected components.
\end{remark}

\begin{definition}[{\cite[Section 5]{HLS21}}]
Let $U$ be a normal integral quasi-projective $R$-scheme and let $B$ be a $\QQ$-Cartier divisor on $U$.
\begin{itemize}
    \item Suppose that $U$ is affine and $B$ is $\QQ$-trivial. Take $X$ and $\barB$ such that $X$ is a normal integral projective $R$-scheme containing $U$ as an open subscheme, and $\barB$ is a $\QQ$-Cartier divisor on $X$ with $\barB \restrict{U} = B$. Then define $\sHLS(\sOm_U, B) \coloneqq \sHLS\mleft(\sOm_X, \barB\mright) \restrict{U}$. This definition is independent of the choice of $X$ and $\barB$.
    \item Let $U = \bigcup_i U_i$ be an open cover by sufficiently small affine open sets. Define $\sHLS(\sOm_U, B)$ by the condition $\sHLS(\sOm_U, B) \restrict{U_i} = \sHLS\mleft(\sOm_{U_i}, B \restrict{U_i}\mright)$ for all i. This definition is independent of the choice of $U_i$.
\end{itemize}
\end{definition}

\begin{remark}
    Let $X$ be a normal integral projective $R$-scheme, let $U$ be an open set in $X$ and let $\Delta \ge 0$ be a $\QQ$-divisor on $X$.
    Suppose that $U$ is regular, $\floor{\Delta \restrict{U}} = 0$, and $\Supp \Delta \restrict{U}$ is SNC. Then $\sHLS(\sO_X,\Delta) \restrict{U} = \sO_X \restrict{U}$ and $\sHLS(\sOm_X,\Delta) \restrict{U} = \sOm_X \restrict{U}$ by \cite[Proposition 4.24]{HLS21}.
\end{remark}

\begin{definition}[{\cite[Definition 6.1]{BMP+23}}]
Suppose that $X$ is a normal integral projective $R$-scheme and $B \ge 0$ is a $\QQ$-divisor on $X$.
Then $(X,B)$ is said to be \emph{globally $\bmplus$-regular} if the natural map $\sO_X \to g_*\sO_{X'}(\floor{g^*B})$ splits for every finite cover $g \colon X' \to X$.
\end{definition}

\begin{remark}[{cf.\ \cite[Corollary 6.11]{BMP+23}}]\label{rmk: B0-gpr}
If $(X,B)$ is globally $\bmplus$-regular, then $\mBMP(X,B,\sOm_X\otimes\sM) = \mH^0(X,\sOm_X\otimes\sM)$ for any line bundle $\sM$ on $X$. In particular, $\sHLS(\sOm_X) = \sOm_X$. Indeed, for every finite cover $g \colon X' \to X$, applying the functor $\sHom_{\sO_X}(\blank, \sOm_X \otimes \sM)$ to the map $\sO_X \to g_*\sO_{X'}(\floor{g^*B})$, we see that the map
\[
    \mH^0\mleft(X',\sO_{X'}\mleft(\dK_{X'}+\ceil{g^*(M-B)}\mright)\mright) \to \mH^0\mleft(X,\sO_X(\dK_X+M)\mright)
\]
is surjective, which gives the assertion.
\end{remark}

\begin{definition}[{\cite[Definition 6.2, Definition 6.9]{MaSc21}}]\label{def: BCMReg}
    Let $(A, \frakm)$ be a complete normal local domain of mixed characteristic $(0,p)$ and of dimension $d$. Let $\Delta \ge 0$ be a $\QQ$-divisor on $\Spec A$ such that $\dK_A + \Delta$ is $\QQ$-Cartier. We may assume that $\dK_A$ is effective. We can write $n(\dK_A + \Delta) = \divisor(f)$ for some $n > 0$ and $f \in A$. We say that $(A, \Delta)$ is \emph{$\text{BCM}$-regular} if for every perfectoid big Cohen-Macaulay $R^\abscl$-algebra $B$, the natural map
    \[
        \mH_\frakm^d(A) \to \mH_\frakm^d(B) \xrightarrow{f^{1/n}} \mH_\frakm^d(B)
    \]
    is injective.
\end{definition}

\begin{remark}\label{rmk: pPlusRegular}
Assume that $(R, \maxmR)$ is a complete DVR of mixed characteristic.
Let $X$ be a normal integral flat quasi-projective $R$-scheme, and let $\Delta \ge 0$ be a $\QQ$-divisor on $X$ such that $\dK_X + \Delta$ is $\QQ$-Cartier.
Suppose that $x \in X$ is a point. Set $A \coloneqq \completion{\sO_{X,x}}$, the completion of $\sO_{X,x}$ at the maximal ideal. Assume either of the following.
\begin{itemize}
    \item The residue characteristic of $A$ is $p$ and $(A, \Delta \restrict{A})$ is $\text{BCM}$-regular.
    \item The residue characteristic of $A$ is $0$ and $(A, \Delta \restrict{A})$ is KLT.
\end{itemize}
Then we have $\sHLS(\sO_X, \Delta)_x = \sO_{X,x}$ by the recent preprint \cite[Theorem B (a), (j)]{BMP+24}.
Note that the $\bmplus$-test ideal $\sHLS(\sO_X, \Delta)$ by \cite{HLS21} is denoted by $\tau_{\mathbf{B}^0}(\sO_X, \Delta)$ in \cite{BMP+24}.
\end{remark}

We generalize \cite[Proposition 4.7]{HLS21}, using the Fujita-type vanishing by Keeler \cite[Theorem 1.5]{Keer03} instead of Serre vanishing:

\begin{lemma}[{cf.~\cite[Proposition 4.7]{HLS21}}]\label{lem: B0-Keeler}
Let $X$ be a normal integral projective $R$-scheme. Suppose that $B$ is a $\QQ$-divisor on $X$, not necessarily effective, and $\sL$ is an ample line bundle on $X$. Then there exists $i_0 \ge 0$ such that
\[
    \mBMP\mleft(X, B, \sOm_X \otimes \sL^i \otimes \sN\mright) = \mH^0\mleft(X, \sHLS(\sOm_X, B) \otimes \sL^i \otimes \sN\mright)
\]
for every $i \ge i_0$ and nef line bundle $\sN$ on $X$.
\end{lemma}

\begin{proof}
We follow the argument of \cite[Proposition 4.7]{HLS21}.

First, we can assume $B \ge 0$: we replace $B$ by $B + H$ for a sufficiently ample divisor $H$, and note that $\sN(H)$ is nef.
It is enough to prove the assertion in the case $i = i_0$ since $\sL^{i-i_0} \otimes \sN$ is nef.
We may thus replace $\sL$ with its high enough power.
For a coherent sheaf $\sF$ on $X$, we write
\[
    \mH^0_*(X, \sF) \coloneqq \bigoplus_{i \ge 0} \mH^0\mleft(X, \sF \otimes \sL^i\mright)\text.
\]
Define $S \coloneqq \mH^0_*(X, \sO_X)$ and $J \coloneqq \bigoplus_{i \ge 0} \mBMP\mleft(X, B, \sOm_X \otimes \sL^i\mright)$.
We note that the $\sO_X$-module associated to $J$ is $\sAssoc{J} \isom \sHLS(\sOm_X, B)$.
\begin{claim}[{cf.~\cite[Chapitre III, Th\'eor\`eme 2.3.1]{EGAIII}}]\label{claim: B0-Keeler}
    Let $M$ be a finitely generated graded $S$-module and write $\sAssoc{M}$ for the $\sO_X$-module associated to $M$.
    Then there exists $i_0 \ge 0$ such that
    \[
        \mleft[M \otimes_S \mH^0_*(X, \sN)\mright]_i \to \mleft[\mH^0_*\mleft(X, \sAssoc{M} \otimes_X \sN\mright)\mright]_i = \mH^0\mleft(X, \sAssoc{M} \otimes \sL^i \otimes \sN\mright)
    \]
    is an isomorphism for every $i \ge i_0$ and nef line bundle $\sN$ on $X$.
\end{claim}
We prove this claim.
Note that the claim is standard if $\sN$ is fixed (\cite[Chapitre III, Th\'eor\`eme 2.3.1]{EGAIII}, \citestacks{0AG7}).
Since $M$ is finitely generated, we can take a surjection $S(a_1) \oplus \dotsm \oplus S(a_s) \toepi M$ where $s \ge 0$ and $a_i \in \ZZ$. Form the exact sequence
\[
    0 \to K \to \bigoplus_{1 \le l \le s} S(a_l) \to M \to 0\text.
\]
Hence, we have the following diagram
\begin{DIFnomarkup}
\[
    \begin{tikzcd}
        &[-1em]
        K \otimes \mH^0_*(\sN) \ar[r] \ar[d, "\kappa"] &
        (\bigoplus_{l} S(a_l)) \otimes \mH^0_*(\sN) \ar[r] \ar[d, "\isom" sloped] &
        M \otimes \mH^0_*(\sN) \ar[r] \ar[d, "\mu"] &[-1em]
        0 \\
        0 \ar[r] &
        \mH^0_*(\sAssoc{K} \otimes \sN) \ar[r] &
        \bigoplus_{l} \mH^0_*(\sN)(a_l) \ar[r] &
        \mH^0_*(\sAssoc{M} \otimes \sN) \ar[r] &
        \mH^1_*(\sAssoc{K} \otimes \sN)
    \end{tikzcd}
\]
\end{DIFnomarkup}with exact rows.
Here we abbreviate $\mH^0_*(X, \sF)$ to $\mH^0_*(\sF)$.
By \cite[Theorem 1.5]{Keer03}, there exists $i_1 \ge 0$ such that $\mH^1_*(\sAssoc{K} \otimes \sN)$ is zero in degree $i \ge i_1$ for every nef line bundle $\sN$ on $X$.
Hence, $\mu$ is surjective in degree $i \ge i_1$ by the snake lemma. Applying the same argument to $K$, we see that there exists $i_0 \ge i_1$ such that $\kappa$ is surjective in degree $i \ge i_0$.
Therefore, $\mu$ is an isomorphism in degree $i \ge i_0$ by the snake lemma, which concludes the proof of \myref{claim: B0-Keeler}.

By \myref{claim: B0-Keeler}, there exists $i_0 \ge 0$ such that the canonical morphism
\[
    \alpha_i \colon \mleft[J \otimes_S \mH^0_*(X, \sN)\mright]_i \xrightarrow{\isom} \mH^0\mleft(X, \sHLS(\sOm_X, B) \otimes \sL^i \otimes \sN\mright)
\]
is isomorphic for every $i \ge i_0$ and nef line bundle $\sN$.
By \cite[Lemma 4.2]{HLS21}, it follows that
\[
    \mH^0\mleft(X, \sHLS(\sOm_X, B) \otimes \sL^i \otimes \sN\mright) = \im \alpha_i \subset \mBMP\mleft(X, B, \sOm_X \otimes \sL^i \otimes \sN\mright)\text,
\]
which completes the proof.
\end{proof}

We remark that the following special case of \myref{lem: B0-Keeler} can be proven without \cite{Keer03}:

\begin{remark}[{\cite[Proposition 4.7]{HLS21}}]\label{rmk: B0-without-Keeler}
Let $X$, $B$ and $\sL$ be as in \myref{lem: B0-Keeler}, and assume $B \ge 0$. Let $\sN$ be a line bundle on $X$. Then there exists $i_0 \ge 0$ such that
\[
    \mBMP\mleft(X, B, \sOm_X \otimes \sL^i \otimes \sN\mright) = \mH^0\mleft(X, \sHLS(\sOm_X, B) \otimes \sL^i \otimes \sN\mright)
\]
for every $i \ge i_0$.

To prove this, since $\sL^i \otimes \sN$ is nef for $i \gg 0$, we may assume that $\sN$ is nef. Hence, the statement becomes a special case of \myref{lem: B0-Keeler}.
In this situation, we can replace Keeler's result \cite[Theorem 1.5]{Keer03} by Serre vanishing in the proof of \myref{lem: B0-Keeler}, since $\sN$ is fixed. Hence, we do not need \cite{Keer03} to show this special case of \myref{lem: B0-Keeler}.
Note that the statement can also be shown as a special case of \myref{lem: RelB0-H0} below.
\end{remark}

The following is a slight generalization of \cite[Lemma 3.10]{HLS21}.

\begin{lemma}\label{lem: B0-et}
Let $X$ be a normal integral projective $R$-scheme. Let $M$ be a divisor on $X$ and let $B$ be a $\QQ$-divisor on $X$. Suppose that $S$ is a finite \'etale domain over $R$.
Write $X_S \coloneqq X \otimes_R S$, $M_S$ and $B_S$ for the base changes by $R \to S$ of $X$, $M$ and $B$. We have
\[
    \mBMP\mleft(X, B, \sO_X\mleft(\dK_X + M\mright)\mright) \otimes_R S = \mBMP\mleft(X_S, B_S, \sO_{X_S}\mleft(\dK_{X_S} + M_S\mright)\mright)\text.
\]
\end{lemma}

\begin{proof}
We can assume that $B \ge 0$ by adding a sufficiently ample divisor to $B$ and $M$.
Set $R' \coloneqq \mH^0\mleft(X, \sO_X\mright)$.
If $R = R'$, the assertion is the same as \cite[Lemma 3.10]{HLS21}.

We reduce the problem to this case, as follows.
Since $R'$ is a domain finite over a complete local domain $R$, we see that $R'$ is a complete local domain from \citestacks{04GH}. Therefore, $R' \otimes_R S$ is a product of finitely many \'etale domains $S'_i$ over $R'$ by \citestacks{04GH}: $R' \otimes_R S \isom \prod_i S'_i$. Then we have
\begin{align*}
    \mBMP\mleft(X, B, \sO_X\mleft(\dK_X + M\mright)\mright) \otimes_R S
    &\isom \mBMP\mleft(X, B, \sO_X\mleft(\dK_X + M\mright)\mright) \otimes_{R'} (R' \otimes_R S) \\
    &\isom \prod_i \mBMP\mleft(X, B, \sO_X\mleft(\dK_X + M\mright)\mright) \otimes_{R'} S'_i
\end{align*}
and
\[
    X_S \isom X \otimes_{R'} (R' \otimes_R S) \isom \coprod_i X_{S'_i}\text.
\]
Thus the assertion follows from the first case.
\end{proof}

\subsection{\texorpdfstring{$\bmplus$}{+}-stable direct images}
We now generalize the $\bmplus$-test ideal \cite{HLS21} to the relative setting, obtaining a mixed characteristic analog of the Frobenius stable direct image \cite{HaXu15}, called the $\bmplus$-stable direct image.
See \myref{def: RelB0} and \myref{def: RelB0mk2}.

We work in the following setting.

\begin{setting}\label{set: RelB0}
Suppose that $(R, \maxmR)$ is a complete Noetherian local domain of residue characteristic $p > 0$.
Let $f \colon X \to Y$ be a surjection between integral projective $R$-schemes with $X$ normal. Define $n$ to be the dimension of the special fiber $Y_{\maxmR}$ of $Y$ over $R$. Write $\nu \colon X^\abscl \to X$ for the natural morphism.
\end{setting}

\begin{definition}[{cf.\ \cite[Definition 4.3]{HLS21}}]\label{def: RelB0}
Let $\sM$ be a line bundle on $X$ and let $B$ be an effective $\QQ$-divisor on $X$. Let $\sL$ be a very ample line bundle on $Y$. For $i > 0$, we define
\[
    \sJ_i \subset f_*(\sOm_X \otimes \sM) \otimes \sL^i
\]
to be the subsheaf globally generated by the submodule of $\mH^0\mleft(Y, f_* (\sOm_X \otimes \sM) \otimes \sL^i\mright)$ corresponding to
\[
    \mBMP\mleft(X, B, \sOm_X \otimes \sM \otimes f^* \sL^i\mright)
    \subset \mH^0\mleft(X, \sOm_X \otimes \sM \otimes f^* \sL^i\mright)
\]
via the projection formula
\[
    \mH^0\mleft(X, \sOm_X \otimes \sM \otimes f^* \sL^i\mright)
    \isom \mH^0\mleft(Y, f_* (\sOm_X \otimes \sM) \otimes \sL^i\mright)\text.
\]
We define the subsheaf
\[
    \relBWM{f}{X}{B}{\sM} \subset f_*(\sOm_X \otimes \sM)
\]
to be $\sJ_i \otimes \sL^{-i}$ for $i \gg 0$, which we call the \emph{$\bmplus$-stable direct image}. This definition does not depend on the choice of the line bundle $\sL$ by the argument similar to \cite[Lemma 4.5]{HLS21}.
We write $\relBWoM{f}{X}{\sM}$ for $\relBWM{f}{X}{0}{\sM}$.
\end{definition}

\begin{remark}\label{rmk: RelB0-easy}
If $Y = \Spec R$ and $f$ is the structure morphism of $X$, we have
\[
    \relBWM{f}{X}{B}{\sM} = \mBMP(X, B, \sOm_X \otimes \sM)\text{.}
\]

If $f$ is finite, we have
\[
    \relBWM{f}{X}{B}{\sM} = f_*(\sHLS(\sOm_X, B) \otimes \sM)\text.
\]
To show this, take an ample line bundle $\sL$ on $Y$ and let $i \gg 0$. Then $f_*(\sHLS(\sOm_X, B) \otimes \sM) \otimes \sL^i$ is globally generated by $\mH^0(Y, f_*(\sHLS(\sOm_X, B) \otimes \sM) \otimes \sL^i)$, and $\relBWM{f}{X}{B}{\sM} \otimes \sL^i$ is globally generated by $\mBMP(X, B, \sOm_X \otimes \sM \otimes f^* \sL^i)$.
We have $\mH^0(Y, f_*(\sHLS(\sOm_X, B) \otimes \sM) \otimes \sL^i) \isom \mBMP(X, B, \sOm_X \otimes \sM \otimes f^* \sL^i)$ by the projection formula and \myref{rmk: B0-without-Keeler} since $f^* \sL$ is ample. Therefore, we obtain the equation.
In particular, if $X = Y$ and $f = \id_X$, we get
\[
    \relBWM{(\id_X)}{X}{B}{\sM} = \sHLS(\sOm_X, B) \otimes \sM\text{.}
\]
\end{remark}

\begin{remark}[{cf.~\cite[Lemma 4.8 (a)]{HLS21}}]\label{rmk: monotone}
If $B \le B'$, then
\[
    \relBWM{f}{X}{B}{\sM} \supset \relBWM{f}{X}{B'}{\sM}\text.
\]
This follows from \myref{def: RelB0} and \cite[Lemma 3.3 (a)]{HLS21}.
\end{remark}

\begin{remark}\label{rmk: RelB0-gpr}
If $(X,B)$ is globally $\bmplus$-regular, we obtain
\[
    \relBWM{f}{X}{B}{\sM} = f_*(\sOm_X \otimes \sM)
\]
by \myref{rmk: B0-gpr}.
\end{remark}

\begin{lemma}[{cf.\ \cite[Proposition 4.7]{HLS21}}]\label{lem: RelB0-H0}
Let $\sM$ be a line bundle on $X$ and let $B \ge 0$ be a $\QQ$-divisor on $X$. Suppose that $\sL$ is an ample line bundle on $Y$. Then we have
\[
    \mBMP\mleft(X, B, \sOm_X \otimes \sM \otimes f^* \sL^i\mright) = \mH^0\mleft(Y, \relBWM{f}{X}{B}{\sM} \otimes \sL^i\mright)
\]
for $i \gg 0$.
\end{lemma}

\begin{proof}
The proof is similar to \cite[Proposition 4.7]{HLS21}.
We may replace $\sL$ with its high enough power (by \myref{lem: RelB0-proj} below, for example). For a coherent sheaf $\sF$ on $Y$, write $\mH^0_*(Y, \sF) \coloneqq \bigoplus_{i \ge 0} \mH^0\mleft(Y, \sF \otimes \sL^i\mright)$.
Define $S \coloneqq \mH^0_*(Y, \sO_Y)$ and $J \coloneqq \bigoplus_{i \ge 0} \mBMP\mleft(X, B, \sOm_X \otimes \sM \otimes f^*\sL^i\mright)$.
$J$ is identified with an $S$-submodule of $\mH^0_*\mleft(Y, f_*(\sOm_X \otimes \sM)\mright)$ by the projection formula.
Note that the $\sO_Y$-module associated to $J$ is $\sAssoc{J} \isom \relBWM{f}{X}{B}{\sM}$. Since the canonical morphism $J \to \mH^0_*(Y, \sAssoc{J})$ is an isomorphism in sufficiently high degrees by \cite[Chapitre III, Th\'eor\`eme 2.3.1]{EGAIII}, we conclude the assertion.
\end{proof}

\begin{lemma}[{cf.\ \cite[Proposition 4.5]{HLS21}}]\label{lem: suffample}
Suppose that $\sM$ is a line bundle on $X$ and $B \ge 0$ is a $\QQ$-divisor on $X$.
Then
\[
    \relBWM{f}{X}{B}{\sM} \otimes \sA
\]
is globally generated by $\mBMP(X, B, \sOm_X \otimes \sM \otimes f^* \sA)$ for every sufficiently ample line bundle $\sA$ on $Y$.
More specifically, for every ample line bundle $\sL$ on $Y$, there exists $i_0 \ge 0$ such that this statement holds for all $\sA$ of the form $\sL^i \otimes \sN$, where $i \ge i_0$ and $\sN$ is a globally generated line bundle on $Y$.
\end{lemma}

\begin{proof}
The proof is similar to \cite[Proposition 4.5]{HLS21}.
Replacing $\sL$ with its high enough power, we may assume that $\sL$ is very ample.
By definition, there exists $i_0 \ge 0$ such that $\relBWM{f}{X}{B}{\sM} \otimes \sL^i$ is globally generated by $\mBMP(X, B, \sOm_X \otimes \sM \otimes f^* \sL^i)$ for every $i \ge i_0$.
For every globally generated line bundle $\sN$ on $Y$, we see from \cite[Lemma 4.2]{HLS21} that $\relBWM{f}{X}{B}{\sM} \otimes \sL^i \otimes \sN$ is globally generated by $\mBMP(X, B, \sOm_X \otimes \sM \otimes f^* (\sL^i \otimes \sN))$, as desired.
\end{proof}

\begin{lemma}\label{lem: RelB0-proj}
Let $\sM$ be a line bundle on $X$ and let $B \ge 0$ be a $\QQ$-divisor on $X$. Let $\sN$ be a line bundle on $Y$. Then
\[
    \relBWM{f}{X}{B}{\sM} \otimes \sN = \relBWM{f}{X}{B}{\sM \otimes f^* \sN}
\]
as subsheaves of $f_*(\sOm_X \otimes \sM) \otimes \sN$.
\end{lemma}

\begin{proof}
We show the assertion by tensoring both sides with a sufficiently ample line bundle $\sA$ on $Y$.
Observe that $\sN \otimes \sA$ is also sufficiently ample.
We thus conclude from \myref{lem: suffample} that both
\[
    \relBWM{f}{X}{B}{\sM} \otimes \sN \otimes \sA \; \text{and} \; \relBWM{f}{X}{B}{\sM \otimes f^* \sN} \otimes \sA
\]
are globally generated by $\mBMP(X, B, \sOm_X \otimes \sM \otimes f^*\sN \otimes f^*\sA)$, completing the proof.
\end{proof}

\begin{lemma}[{cf.\ \cite[Lemma 4.8 (b)]{HLS21}}]\label{lem: RelB0-N}
Let $\sM$ be a line bundle on $X$, let $B \ge 0$ be a $\QQ$-divisor on $X$, and let $F \ge 0$ be a Cartier divisor on $X$. Then
\[
    \relBWM{f}{X}{B}{\sO_X(-F) \otimes \sM} = \relBWM{f}{X}{B + F}{\sM}
\]
as subsheaves of $f_*(\sOm_X \otimes \sM)$.
\end{lemma}

\begin{proof}
Let $\sL$ be an ample line bundle on $Y$ and let $i \gg 0$.
The left- and right-hand side twisted by $\sL^i$ are globally generated by $\mBMP(X, B, \sOm_{X} \otimes \sO_X(-F) \otimes \sM)$ and $\mBMP(X, B + F, \sOm_{X} \otimes \sM)$, respectively.
By \cite[Lemma 3.3 (c)]{HLS21}, we obtain the equality.
\end{proof}

\begin{proposition}\label{prop: B0-composition}
Let $g \colon Y \to Z$ be a surjection to an integral projective $R$-scheme $Z$. Suppose that $\sM$ is a line bundle on $X$ and $B \ge 0$ is a $\QQ$-divisor on $X$. Then
\[
    \relBWM{(g \circ f)}{X}{B}{\sM} \subset g_* \relBWM{f}{X}{B}{\sM}
\]
as subsheaves of $(g \circ f)_*(\sOm_X \otimes \sM)$.
\end{proposition}

\begin{proof}[Proof ({cf.\ \cite[Proposition 5.1]{HLS21}})]
Let $\sL$ be a sufficiently ample line bundle on $Z$.
Define the subsheaf $\sJ \subset f_*(\sOm_X \otimes \sM)$ by the condition that $\sJ \otimes g^*\sL$ is globally generated by
\[
    \mBMP(X, B, \sOm_X \otimes \sM \otimes (g \circ f)^*\sL)\text.
\]
It follows from the definition of $\relBWM{(g \circ f)}{X}{B}{\sM}$ that
\[
    \relBWM{(g \circ f)}{X}{B}{\sM} \otimes \sL \subset g_*(\sJ \otimes g^*\sL) = g_*\sJ \otimes \sL\text.
\]

Suppose that $\sA$ is a sufficiently ample line bundle on $Y$. We have the map
\[
    \mBMP\mleft(X, B, \sOm_X \otimes \sM \otimes (g \circ f)^*\sL\mright) \otimes \mH^0(X, f^*\sA) \to \mBMP(X, B, \sOm_X \otimes \sM \otimes (g \circ f)^*\sL \otimes f^*\sA)
\]
by \cite[Lemma 4.2]{HLS21}. Hence we see from \myref{lem: suffample} that
\[
    (\sJ \otimes g^*\sL) \otimes \sA \subset \relBWM{f}{X}{B}{\sM} \otimes g^*\sL \otimes \sA\text.
\]
We thus conclude that
\[
    \relBWM{(g \circ f)}{X}{B}{\sM} \subset g_*\sJ \subset g_*\relBWM{f}{X}{B}{\sM}\text.
\]
\end{proof}

\begin{corollary}\label{cor: RelB0-incl}
For a line bundle $\sM$ on $X$ and a $\QQ$-divisor $B \ge 0$ on $X$, we have
\[
    \relBWM{f}{X}{B}{\sM} \subset f_*(\sHLS(\sOm_X, B) \otimes \sM)\text.
\]
\end{corollary}

\begin{proof}
This follows from \myref{prop: B0-composition} and \myref{rmk: RelB0-easy}.
\end{proof}

\begin{definition}[{cf.\ \cite[Definition 4.14, 4.15]{HLS21}}]\label{def: RelB0mk2}
Let $\sM$ be a line bundle on $X$.
\begin{itemize}
\item Let $B$ be a $\QQ$-divisor on $X$ (not necessarily effective). We set
    \[
        \relBWM{f}{X}{B}{\sM} \coloneqq \relBWM{f}{X}{B + H}{\sM(H)} \subset f_*\mleft(\sOm_X \otimes \sM \otimes \rK(X)\mright)
    \]
    for a Cartier divisor $H$ satisfying $B + H \ge 0$. By \myref{lem: RelB0-N}, this definition is independent of the choice of $H$.
\item Let $\Delta \ge 0$ be a $\QQ$-divisor on $X$. We set
    \begin{align*}
        \relBOM{f}{X}{\Delta}{\sM} &\coloneqq \relBWM{f}{X}{\dK_X + \Delta}{\sM} \\
        &\subset f_*\mleft(\sHLS(\sOm_X, \dK_X + \Delta) \otimes \sM\mright)
        \isom f_*(\sHLS(\sO_X, \Delta) \otimes \sM) \\
        &\subset f_*\sM
    \end{align*}
    where the first inclusion follows from \myref{cor: RelB0-incl}.
    This subsheaf of $f_*\sM$ is independent of the choice of a canonical divisor $\dK_X$, as follows (cf.~\cite[Lemma 4.17]{HLS21}). For a Cartier divisor $H$ such that $\dK_X + H \ge 0$ and a sufficiently ample line bundle $\sL$, $\relBOM{f}{X}{\Delta}{\sM} \otimes \sL$ is globally generated by $\mBMP(X, \dK_X + H + \Delta, \sOm_X(H) \otimes \sM \otimes f^* \sL)$. We have the diagram
    \begin{DIFnomarkup}
    \[
        \begin{tikzcd}
            \mBMP(X, \Delta, \sM \otimes f^* \sL) \arrow[r, "\isom"] \arrow[d, phantom, "\subset" sloped]
            & \mBMP(X, \dK_X + H + \Delta, \sOm_X(H) \otimes \sM \otimes f^* \sL) \arrow[d, phantom, "\subset" sloped] \\
            \mH^0(X, \sM \otimes f^* \sL) \arrow[r, hook]
            & \mH^0(X, \sOm_X(H) \otimes \sM \otimes f^* \sL)
        \end{tikzcd}
    \]
    \end{DIFnomarkup}by \cite[Lemma 3.3 (c)]{HLS21}. Hence, $\relBOM{f}{X}{\Delta}{\sM} \otimes \sL$ is the subsheaf of $f_*\sM \otimes \sL$ globally generated by $\mBMP(X, \Delta, \sM \otimes f^* \sL)$, which is independent of the choice of the canonical divisor $\dK_X$. Therefore, $\relBOM{f}{X}{\Delta}{\sM}$ is independent of the choice of a canonical divisor $\dK_X$.
\end{itemize}
\end{definition}

\begin{remark}\label{rmk: non-eff}
    Note that all the lemmas above  concerning $\bmplus$-stable direct images (from \myref{rmk: RelB0-easy} to \myref{cor: RelB0-incl}) are naturally generalized to the case of a possibly non-effective $\QQ$-divisor $B$.

    For example, \myref{prop: B0-composition} extends as follows. Let $g \colon Y \to Z$ be a surjection to an integral projective $R$-scheme $Z$. Suppose that $\sM$ is a line bundle on $X$ and $B$ is a $\QQ$-divisor on $X$, not necessarily effective. Then
    \[
        \relBWM{(g \circ f)}{X}{B}{\sM} \subset g_* \relBWM{f}{X}{B}{\sM}
    \]
    as subsheaves of $(g \circ f)_*(\sOm_X \otimes \sM \otimes \rK(X))$. Note that the only difference (except for the non-effectivity of $B$) is the last ambient sheaf.

    This extension can be derived from \myref{prop: B0-composition} by replacing $B$ with $B+H$ and $\sM$ with $\sM(H)$ for a sufficiently ample divisor $H$ on $X$.
    We can similarly generalize the other lemmas for $\bmplus$-stable direct images.
\end{remark}

\begin{remark}\label{rmk: RelB0-M}
For a line bundle $\sM = \sO_X(M)$ and a $\QQ$-divisor $B$ on $X$, we have
\[
    \relBWM{f}{X}{B}{\sM} = \relBW{f}{X}{B - M}\text.
\]
Thus, the use of $\sM$ in \myref{def: RelB0mk2} is only for notational convenience.
\end{remark}

\begin{proposition}\label{prop: RelB0-Keeler}
Suppose that $B$ is a $\QQ$-divisor on $X$ and $\sM$ is an $f$-ample line bundle on $X$. Then for all $m \gg 0$,
\[
    \relBWM{f}{X}{B}{\sM^m} = f_*\mleft(\sHLS(\sOm_X, B) \otimes \sM^m\mright)\text.
\]
\end{proposition}

\begin{proof}
Let $\sL$ be a very ample line bundle on $Y$. Since $\sM$ is $f$-ample, $\sM \otimes f^*\sL^l$ is ample for some $l > 0$. By \myref{lem: B0-Keeler}, there exists an integer $m_0$ such that
\[
    \mBMP\mleft(X, B, \sOm_X \otimes \mleft(\sM \otimes f^*\sL^l\mright)^m \otimes \sN'\mright) = \mH^0\mleft(X, \sHLS(\sOm_X, B) \otimes \mleft(\sM \otimes f^*\sL^l\mright)^m \otimes \sN'\mright)
\]
for every $m \ge m_0$ and nef line bundle $\sN'$ on $X$. Fix $m \ge m_0$. For every $i \ge 0$, by setting $\sN' = f^*\sL^i$, we have
\begin{align*}
    \mBMP\mleft(X, B, \sOm_X \otimes \sM^m \otimes f^*\sL^{ml+i}\mright) &= \mH^0\mleft(X, \sHLS(\sOm_X, B) \otimes \sM^m \otimes f^*\sL^{ml+i}\mright) \\
    &\isom \mH^0\mleft(Y, f_*\mleft(\sHLS(\sOm_X, B) \otimes \sM^m\mright) \otimes \sL^{ml+i}\mright)\text.
\end{align*}
Thus, the assertion follows, since the left-hand side generates $\relBWM{f}{X}{B}{\sM^m} \otimes \sL^{ml+i}$ by definition, and the right-hand side generates $f_*\mleft(\sHLS(\sOm_X, B) \otimes \sM^m\mright) \otimes \sL^{ml+i}$ by Serre vanishing.
\end{proof}

\begin{lemma}\label{lem: RelB0-et}
Let $B$ be a $\QQ$-divisor on $X$ and let $\sM$ be a line bundle on $X$. Suppose that $S$ is a finite \'etale domain over $R$. Write $X_S \coloneqq X \otimes_R S$, $f_S$, $B_S$, and $\sM_S$ for the base changes by $R \to S$. Then
\[
    \relBWM{f}{X}{B}{\sM} \otimes_R S = \relBWM{(f_S)}{{X_S}}{B_S}{\sM_S}\text.
\]
\end{lemma}

\begin{proof}
By \myref{rmk: RelB0-M}, we can assume that $\sM = \sO_X$.
Let $\sL$ be an ample line bundle on $Y$, so that $\sL_S$ is ample on $Y_S$. For $i \gg 1$, the left-hand side is given by $\im(\mBMP(X, B, \sOm_X \otimes f^* \sL^i) \otimes \sL^{-i} \to f_*\sOm_X)$, while the right-hand side is given by $\im(\mBMP(X_S, B_S, \sOm_{X_S} \otimes f_S^*\sL_S^i) \otimes \sL_S^{-i} \to f_{S,*}\sOm_{X_S})$.
As $R \to S$ is \'etale (hence flat), we conclude by \myref{lem: B0-et}.
\end{proof}

While \cite{HLS21} defined the $\bmplus$-test ideal more generally for a quasi-projective $R$-scheme, we do not give a similar generalization of the $\bmplus$-stable direct image.
We only give lemmas used in this paper.

\begin{proposition}\label{prop: RelB0-fmu}
Let $B$ be a $\QQ$-Cartier divisor on $X$ and let $\mu \colon X' \to X$ be a birational morphism from a normal integral projective $R$-scheme $X'$. Then
\[
    \relBW{f}{X}{B} = \relBW{(f \circ \mu)}{X'}{\mu^* B}\text{.}
\]
\end{proposition}

\begin{proof}
Let $\sL$ be an ample line bundle on $Y$. For $i \gg 0$, the left- and right-hand side tensored with $\sL^i$ is generated by $\mBMP(X, B, \sOm_X \otimes f^*\sL^i)$ and $\mBMP(X', \mu^*B, \sOm_{X'} \otimes (f \circ \mu)^*\sL^i)$, respectively.
Hence the claim follows from \cite[Proposition 3.9]{HLS21}.
\end{proof}

\begin{lemma}\label{lem: RelB0-Op}
Let $B$ and $B'$ be $\QQ$-divisors on $X$. Suppose that $V$ is an open set in $Y$ and $U \coloneqq f^\inv(V) \subset X$. If $B \restrict{U} = B' \restrict{U}$, then we have
\[
    \relBW{f}{X}{B} \restrict{V} = \relBW{f}{X}{B'} \restrict{V}\text{.}
\]
\end{lemma}

\begin{proof}[\proofname{} ({cf.\ \cite[Theorem 5.3]{HLS21}})]
Since the statement is local on $V$, we may assume that $V$ is the complement of the support of an ample effective Cartier divisor $H$ on $Y$. By assumption, we have
\[
    -i f^*H \le B' - B \le i f^*H
\]
for some $i \ge 0$. Therefore, we obtain
\begin{align*}
    \relBW{f}{X}{B} \restrict{V}
    &= \relBW{f}{X}{B} \otimes \sO_Y(i H) \restrict{V} \\
    &= \relBW{f}{X}{B - i f^*H} \restrict{V} \\
    &\supseteq \relBW{f}{X}{B'} \restrict{V}
\end{align*}
by \myref{lem: RelB0-proj}, \myref{lem: RelB0-N} and \myref{rmk: monotone}. By symmetry, we get the reverse inclusion.
\end{proof}

\section{Global generation for canonical bundles}
In this section, we show \myref{main: relB0}, a global generation theorem for the $\bmplus$-stable direct image. It is the key ingredient of the proof of \myref{main: main}.
We also consider the global generation for the direct image of a canonical sheaf.

We still work in \myref{set: RelB0} throughout this section.

\begin{proposition}\label{prop: RelGg-Van}
Let $B$ be a $\QQ$-Cartier divisor on $X$ and let $\sL$ be a globally generated ample line bundle on $Y$. Let $\sM$ be a line bundle on $X$. Set $d \coloneqq \dim X$. Assume
\[
    \fHRRGa{i}{\maxmR}(X^\abscl, \sO_{X^\abscl}(\nu^*B) \otimes \nu^*\sM^\inv \otimes \nu^*f^*\sL^{-a}) = 0
\]
for all $i < d$ and $a \ge 0$. Then
\[
    \relBWM{f}{X}{B}{\sM} \otimes \sL^n
\]
is globally generated by $\mBMP\mleft(X, B, \sOm_X \otimes \sM \otimes f^*\sL^n\mright)$.
\end{proposition}

\begin{proof}
The proof is based on that of \cite[Theorem 6.1]{HLS21}, which is, in turn, inspired by the arguments found in \cite{ScTu14, LazII}.

We may assume $B \ge 0$ by \myref{def: RelB0mk2}.

Recall that $n \coloneqq \dim Y_{\maxmR}$.
If $R / \maxmR$ is an infinite field, there are global generators $s_0, \dots, s_n \in \mH^0(Y,\sL)$ of the globally generated line bundle $\sL$. Suppose that $R / \maxmR$ is finite. By \myref{lem: B0-et} and \myref{lem: RelB0-et}, we may replace $R$ with a finite \'etale extension of it. Thus, we have global generators $s_0, \dots, s_n \in \mH^0(Y,\sL)$ also in this case.

We have the Koszul complex for $s_0, \dots, s_n$:
\[
    0 \to \sF_{n+1} \to \dots \to \sF_1 \to \sF_0 \to 0, \quad \sF_i \coloneqq \mleft(\sL^{-i}\mright)^{\powoplus \binom{n+1}{i}}\text.
\]
This complex is an exact sequence consisting of locally free sheaves. Let $m \ge n+1$. Applying the functor $\sHom(\nu^*(\sM \otimes f^*(\blank \otimes \sL^m)), \sO_{X^\abscl}(\nu^*B))$, we get the exact sequence
\begin{gather*}
    0 \leftarrow \sG_{n+1} \leftarrow \dots \leftarrow \sG_1 \leftarrow \sG_0 \leftarrow 0, \\
    \sG_i \coloneqq {\sN_i}^{\powoplus \binom{n+1}{i}}, \quad \sN_i \coloneqq \sO_{X^\abscl}(\nu^*B)\otimes\nu^*\sM^\inv\otimes \nu^*f^*\sL^{-(m-i)}\text,
\end{gather*}
By assumption, $\fHRRGa{d - i + 1}{\maxmR}(X^\abscl, \sG_i) = 0$ for $2 \le i \le n+1$, and hence a diagram chase shows that the morphism
\[
    \fHRRGa{d}{\maxmR}\mleft(X^\abscl, \sG_1\mright) \leftarrow \fHRRGa{d}{\maxmR}\mleft(X^\abscl, \sG_0\mright)
\]
is injective (see \cite[Appendix B.1]{LazI}). Consider the diagram below:
\begin{DIFnomarkup}
\[
    \begin{tikzcd}
        \fHRRGa{d}{\maxmR}\mleft(X^\abscl, \sN_1\mright)^{\powoplus n+1}
        & \arrow[l, hook'] \fHRRGa{d}{\maxmR}\mleft(X^\abscl, \sN_0\mright) \\
        \arrow[u, "\alpha"] \fHRRGa{d}{\maxmR}\mleft(X, \sM^\inv \otimes f^*\sL^{-(m-1)}\mright)^{\powoplus n+1}
        & \arrow[l] \arrow[u, "\beta"] \fHRRGa{d}{\maxmR}\mleft(X, \sM^\inv \otimes f^*\sL^{-m}\mright)
    \end{tikzcd}_{\textstyle .}
\]
\end{DIFnomarkup}It follows that the map $\im \alpha \leftarrow \im \beta$ is injective. Applying Matlis duality, we deduce from \cite[Lemma 4.8]{BMP+23} that
\[
    \mBMP\mleft(X, B, \sOm_X \otimes \sM \otimes f^*\sL^{m-1}\mright) \otimes \mH^0(Y, \sL) \to \mBMP\mleft(X, B, \sOm_X \otimes \sM \otimes f^*\sL^m\mright)
\]
is surjective.

Since $m \ge n+1$ is arbitrary, we see that
\[
    \mBMP\mleft(X, B, \sOm_X \otimes \sM \otimes f^*\sL^n\mright) \otimes \mH^0(Y, \sL^t) \to \mBMP\mleft(X, B, \sOm_X \otimes \sM \otimes f^*\sL^{n+t}\mright)
\]
is surjective for every $t \ge 0$, and hence that
\[
    \relBWM{f}{X}{B}{\sM} \otimes \sL^n
\]
is globally generated by $\mBMP(X, B, \sOm_X \otimes \sM \otimes f^*\sL^n)$.
\end{proof}

\myref{prop: RelGg-Van} leads to the following, which is a relativization of \cite[Theorem 6.1]{HLS21} and a mixed characteristic analog of \cite[Theorem 1.2]{Ejir23a}.

\begin{corollary}\label{cor: RelGg}
Let $B$ be a $\QQ$-Cartier divisor on $X$ with $-B$ big and semiample. Let $\sL$ be a globally generated ample line bundle on $Y$. Then
\[
    \relBW{f}{X}{B} \otimes \sL^n
\]
is globally generated by $\mBMP(X, B, \sOm_X \otimes f^*\sL^n)$.
\end{corollary}

\begin{proof}
Combining \myref{prop: RelGg-Van} with Bhatt's vanishing theorem (\myref{thm: Bhatt}) gives the conclusion.
\end{proof}

We now prove \myref{main: relB0}, generalizing \myref{cor: RelGg}.

\begin{theorem}[\myref{main: relB0}]\label{thm: RelGgOp}
We work in \myref{set: RelB0}. Let $B$ be a $\QQ$-Cartier divisor on $X$. Suppose that $V \subset Y$ is an open set, $U \coloneqq f^\inv(V)$, and $-B$ is semiample over $U$ and big. Let $\sL$ be a globally generated ample line bundle on $Y$. Then
\[
    \relBW{f}{X}{B} \otimes \sL^n
\]
is globally generated over $V$ by $\mBMP(X, B, \sOm_X \otimes f^*\sL^n)$.
\end{theorem}

\begin{remark}
\myref{thm: RelGgOp} and \myref{main: relB0} might appear different, but the former readily implies the latter.
Indeed, set $B = \dK_X + \Delta - M$ and notice \myref{rmk: RelB0-M} and \myref{def: RelB0mk2}.
\end{remark}

\begin{proof}
Let $j \gg 0$. We take the normalized blowup $\mu \colon X' \to X$ along $\Bs(-jB)$; we see that
\[
    \mu^*(-jB) = M' + F'\text,
\]
where $M'$ is the free part and $F'$ is the fixed part. Since $-B$ is big, we see that $M'$ is also big. It follows from \myref{cor: RelGg} that
\[
    \relBW{(f \circ \mu)}{X'}{-\frac{M'}{j}} \otimes \sL^n
\]
is globally generated by $\mBMP(X', -M'/j, \sOm_{X'} \otimes (f \circ \mu)^*\sL^n)$.
Since $-M'/j \ge \mu^*B$, by \cite[Proposition 3.9, Lemma 3.3 (a)]{HLS21}, we have the diagram
\begin{DIFnomarkup}
\[
    \begin{tikzcd}
        \mBMP\mleft(X', -\frac{M'}{j}, \sOm_{X'} \otimes (f \circ \mu)^*\sL^n\mright) \ar[r, hook] \ar[d, phantom, sloped, "\subset"]
        & \mBMP(X, B, \sOm_{X} \otimes f^*\sL^n) \ar[d, phantom, sloped, "\subset"] \\
        \mH^0(X', \sOm_{X'} \otimes (f \circ \mu)^*\sL^n \otimes \rK(X')) \ar[r, hook]
        & \mH^0(X, \sOm_{X} \otimes f^*\sL^n \otimes \rK(X))
    \end{tikzcd}_{\textstyle .}
\]
\end{DIFnomarkup}We have morphisms
\begin{alignat*}{2}
    \relBW{(f \circ \mu)}{X'}{-\frac{M'}{j}} \otimes \sL^n &\tomono{} &&\relBW{(f \circ \mu)}{X'}{-\frac{M'}{j}-\frac{F'}{j}} \otimes \sL^n \\
    &\to{} &&\relBW{f}{X}{B} \otimes \sL^n\text{,}
\end{alignat*}
which are isomorphic over $V$ by \myref{lem: RelB0-Op} and \myref{prop: RelB0-fmu}.
Thus we obtain the assertion.
\end{proof}

We give a variant of \myref{prop: RelGg-Van} based on a similar argument:
\begin{proposition}\label{prop: RelGg-Variant}
Let $\sL$ be a globally generated ample line bundle on $Y$. Assume that for all $i > 0$ and $a \ge 0$,
\[
    \varprojlim_{g \colon X' \to X} \mH^i\mleft(Y, f_*g_*\sOm_{X'} \otimes \sL^{a+1}\mright) = 0\text{,}
\]
where $g \colon X' \to X$ runs over all generically finite projective morphisms from a normal integral scheme $X'$. Then
\[
    \relBWo{f}{X} \otimes \sL^{n+1}
\]
is globally generated by $\mBMP(X, \sOm_X \otimes f^*\sL^{n+1})$.
\end{proposition}

\begin{proof}
The proof is similar to that of \myref{prop: RelGg-Van}.
We can assume that there are global generators $s_0, \dots, s_n \in \mH^0(Y,\sL)$ of $\sL$, by using \myref{lem: B0-et} and \myref{lem: RelB0-et}.
Form the exact sequence
\[
    0 \to \sG_{g,n+1} \to \dots \to \sG_{g,1} \to \sG_{g,0} \to 0, \quad \sG_{g,i} \coloneqq \mleft(f_*g_*\sOm_{X'} \otimes \sL^{m+1-i}\mright)^{\powoplus \binom{n+1}{i}}
\]
for every $g \colon X' \to X$ and $m \ge n+1$.
By assumption, we have $\varprojlim_g \mH^{i-1}(Y, \sG_{g,i}) = 0$ for every $2 \le i \le n + 1$.
By diagram chasing, we see that the morphism
\[
    \varprojlim_g \mH^0\mleft(Y, \sG_{g,1}\mright) \to \varprojlim_g \mH^0\mleft(Y, \sG_{g,0}\mright)
\]
is surjective (see \cite[Appendix B.1]{LazI}), noting the exactness of the inverse limit \cite[Lemma 4.10]{BMP+23}.
In the diagram
\begin{DIFnomarkup}
\[
    \begin{tikzcd}
        \varprojlim_g \mH^0\mleft(X', \sOm_{X'} \otimes g^*f^*\sL^m\mright)^{\powoplus n+1} \arrow[r, two heads] \arrow[d, "\alpha"]
        & \varprojlim_g \mH^0\mleft(X', \sOm_{X'} \otimes g^*f^*\sL^{m+1}\mright) \arrow[d, "\beta"] \\
        \mH^0\mleft(X, \sOm_{X} \otimes f^*\sL^m\mright)^{\powoplus n+1} \arrow[r]
        & \mH^0\mleft(X, \sOm_{X} \otimes f^*\sL^{m+1}\mright)
    \end{tikzcd}_{\textstyle ,}
\]
\end{DIFnomarkup}the upper row is surjective, and hence $\im \alpha \to \im \beta$ is surjective.
We note that $\mBMP$ can be computed using alterations \cite[Corollary 4.13]{BMP+23}, and we have
\[
    \text{$\im \alpha = \mBMP(X, \sOm_X \otimes f^*\sL^{m})^{\powoplus n+1}$ and $\im \beta = \mBMP(X, \sOm_X \otimes f^*\sL^{m+1})$}
\]
by \cite[Lemma 4.8 (d)]{BMP+23}. Since $m \ge n+1$ is arbitrary,
\[
    \mBMP\mleft(X, \sOm_X \otimes f^*\sL^{n+1}\mright) \otimes \mH^0(Y, \sL^t) \to \mBMP\mleft(X, \sOm_X \otimes f^*\sL^{n+1+t}\mright)
\]
is surjective for every $t \ge 0$. Hence the assertion follows.
\end{proof}

We observe that the vanishing in the assumption of \myref{prop: RelGg-Variant} holds in equal characteristic $p \ge 0$:
\begin{remark}
Suppose that $X$ is a normal projective variety over a perfect field of characteristic $p \ge 0$.
Then for all $i > 0$ and $a \ge 0$,
\[
    \varprojlim_{g \colon X' \to X} \mH^i\mleft(Y, f_*g_*\sOm_{X'} \otimes \sL^{a+1}\mright) = 0\text{,}
\]
where $g \colon X' \to X$ runs over all generically finite morphisms from a normal projective variety $X'$.

It suffices to show that for every generically finite morphism $g_1 \colon X_1 \to X$ from a normal projective variety $X_1$, there exists a generically finite morphism $g \colon X' \to X_1$ from a normal projective variety $X'$ such that
\[
    \mH^i\mleft(Y, f_*g_{1,*}g_*\sOm_{X'} \otimes \sL^{a+1}\mright) = 0\text.
\]
Replacing $X$ with $X_1$, we may assume that $X_1 = X$.

If $k$ is of positive characteristic, then we set $g$ to be the iterated Frobenius $\Fr^e$ ($e \gg 0$) and obtain the equation by Serre vanishing.
If $k$ is of characteristic zero, then we take a desingularization $g \colon X' \to X$ and use Koll\'ar's vanishing \cite[Theorem 2.1]{Koll86I}.
\end{remark}

\subsection{Global generation of \texorpdfstring{$f_*\sOm_X \otimes \sL^{l}$}{f\_* O\_X otimes L\^{}l}}
Using \myref{prop: RelGg-Van}, we investigate the global generation of the direct image sheaf $f_*\sOm_X \otimes \sL^{l}$.
To this end, we also study conditions on the singularities.

\begin{corollary}\label{cor: gg-gpr}
We work in \myref{set: RelB0}. Assume that $f$ is generically finite and $X$ is globally $\bmplus$-regular. Let $\sL$ be a globally generated ample line bundle on $Y$. Then
\[
    f_*\sOm_X \otimes \sL^l
\]
is globally generated for $l \ge n+1$.
\end{corollary}

\begin{proof}
Since $f$ is generically finite, $f^*\sL$ is big and semiample. The conclusion follows from \myref{cor: RelGg} and \myref{rmk: RelB0-gpr}.
\end{proof}

\begin{corollary}\label{cor: gg-alt-tau}
We work in \myref{set: RelB0}. Assume that $f$ is generically finite. Let $\sL$ be a globally generated ample line bundle on $Y$. Define $V$ to be the maximal open subset of $Y$ satisfying $\relBWo{f}{X} \restrict{V} = f_*\sHLS(\sOm_X) \restrict{V}$. Then $V$ contains the locus over which $f$ is finite, and
\[
    f_*\sHLS(\sOm_X) \otimes \sL^l
\]
is globally generated over $V$ for $l \ge n+1$.
\end{corollary}

\begin{proof}
Let $X \xrightarrow{\mu} X' \xrightarrow{g} Y$ be the Stein factorization of $f$. Then we see that $\relBWo{f}{X} = \relBWo{g}{X'} = g_*\sHLS(\sOm_{X'})$ by \myref{prop: RelB0-fmu} and \myref{rmk: RelB0-easy}.
Let $V' \subset Y$ be the locus over which $f$ is finite.
Since $g_*\sHLS(\sOm_{X'}) \restrict{V'} = f_*\sHLS(\sOm_X) \restrict{V'}$, we have $V' \subset V$.

By \myref{cor: RelB0-incl}, we have the inclusion
\[
    \relBWo{f}{X} \otimes \sL^l \subset f_*\sHLS(\sOm_X) \otimes \sL^l\text,
\]
which is an isomorphism on $V$.
By \myref{cor: RelGg}, $\relBWo{f}{X} \otimes \sL^l$ is globally generated over $V$, and thus the assertion follows.
\end{proof}

To deduce the global generation of $f_*\sOm_X \otimes \sL^l$ from \myref{cor: gg-alt-tau}, we explore conditions of the singularities that imply $\sHLS(\sOm_X) = \sOm_X$.

In the case where $R$ is a DVR, we obtain the following result using the recent preprint \cite{BMP+24}.

\begin{remark}\label{rmk: gg-alt-dvr}
We work in \myref{set: RelB0}. Assume that $R$ is a complete DVR of mixed characteristic $(0, p > 0)$, $X \to \Spec R$ is flat, and $f$ is generically finite. Let $\sL$ be a globally generated ample line bundle on $Y$, and set $d = \dim X$.
Define $V$ to be the maximal open subset of $Y$ satisfying $\relBWo{f}{X} \restrict{V} = f_*\sHLS(\sOm_X) \restrict{V}$. Suppose that for each closed point $x \in X$, there exists a rational number $0 < \epsilon \ll 1$ such that the natural map
\[
    \mH^d_{\frakm_x}(\sO_{X,x}) \to \mH^d_{\frakm_x}(\sO_{X,x}^\abscl) \xrightarrow{p^\epsilon} \mH^d_{\frakm_x}(\sO_{X,x}^\abscl)
\]
is injective. Then $V$ contains the locus over which $f$ is finite, and
\[
    f_*\sOm_X \otimes \sL^l
\]
is globally generated over $V$ for $l \ge n+1$.

To show the assertion, it is enough to prove $\sHLS(\sOm_X) = \sOm_X$ by \myref{cor: gg-alt-tau}.
Let $x \in X$ be a closed point and set $A \coloneqq \sO_{X,x}$.
By \cite[Proposition 4.29]{BMP+23} and the assumption on the local cohomology, we have $\sOm_{\completion{A}} = \mBMP(\Spec \completion{A}, \epsilon \divisor(p), \sOm_{\completion{A}})$ for $0 < \epsilon = 1/p^e \ll 1$.
By \cite[Proposition 6.17, 6.19]{BMP+24}, we obtain $\sHLS(\sOm_X, \epsilon \divisor(p)) \otimes \completion{A} = \mBMP(\Spec \completion{A}, \epsilon \divisor(p), \sOm_{\completion{A}})$.
Note that the left- and right-hand side is denoted by $\tau_{\mBMP}(\sOm_X, \epsilon \divisor(p)) \otimes \completion{A}$ and $\tau_+(\sOm_{\completion{A}}, p^\epsilon)$ in \cite[Definition 2.9, 6.16]{BMP+24}, respectively.
It follows that $\sOm_{X,x} = \sHLS(\sOm_X, \epsilon \divisor(p)) \otimes \sO_{X,x} \subset \sHLS(\sOm_X) \otimes \sO_{X,x}$.
Since $x \in X$ is an arbitrary closed point, we obtain $\sHLS(\sOm_X) = \sOm_X$.
\end{remark}

We also discuss the case where $R$ is not a DVR. We recall the following definition.

\begin{definition}
Let $(A, \frakm)$ be a $d$-dimensional excellent normal local ring. Then $A$ is said to
\begin{itemize}
\item be a \emph{splinter} if every finite extension $A \tomono B$ splits.
\item have a \emph{regular finite cover} if there exists a finite extension $A \tomono B$ to a regular ring $B$.
\item be a \emph{finite summand singularity} if $A$ is a splinter and has a regular finite cover. This definition is equivalent to \cite[Definition 5.6]{BMP+23} by the direct summand theorem \cite{Andr18}.
\end{itemize}
\end{definition}

\begin{remark}
\cite[Theorem 5.8]{BMP+23} implies that if every stalk of $X$ is a finite summand singularity, then $\sHLS(\sOm_X) = \sOm_X$.
\end{remark}

We now introduce two singularity conditions for a $d$-dimensional normal local ring $(A, \frakm)$.

\begin{condition}\label{cond: fincov}
Let $(A, \frakm)$ be a $d$-dimensional excellent normal local ring. The map $\mH^d_\frakm(A) \to \mH^d_\frakm(A^\abscl)$ is injective and $A$ has a regular finite cover.
\end{condition}

\begin{condition}\label{cond: regseq}
Let $(A, \frakm)$ be a $d$-dimensional excellent normal local ring. There exist $r \in \ZZge{0}$ and a regular sequence $a_1,\dots,a_r \in A$ such that $A / (a_1,\dots,a_r)$ satisfies \myref{cond: fincov}.
\end{condition}

\begin{remark}
For a $d$-dimensional excellent normal local ring $(A, \frakm)$, we see that
\begin{align*}
    \text{$A$ is regular}
    &\implies \text{$A$ is a finite summand singularity} \\
    &\implies \text{$A$ satisfies \myref{cond: fincov}} \\
    &\implies \text{$A$ satisfies \myref{cond: regseq}} \\
    &\implies \text{$\mH^d_{\frakm}(A) \to \mH^d_{\frakm}(A^\abscl)$ is injective.}
\end{align*}
Here, the first implication follows from the direct summand theorem \cite{Andr18} and the last implication follows from the proof of \cite[Proposition 3.4]{MaSc21}.
\end{remark}

We show in \myref{thm: sing} that $\sHLS(\sOm_X) = \sOm_X$ holds if every stalk of $X$ satisfies \myref{cond: regseq}.

\begin{lemma}\label{lem: B0-fin}
Let $g \colon U \to V$ be a finite surjection between normal integral quasi-projective $R$-schemes. Let $B \ge 0$ be a $\QQ$-divisor on $V$. Then
\[
    \Tr \mleft(g_*\sHLS\mleft(\sOm_U, g^*B\mright)\mright) = \sHLS\mleft(\sOm_V, B\mright)\text{,}
\]
where $\Tr \colon g_*\sOm_U \isom \sHom_V\mleft(g_*\sO_U, \sOm_V\mright) \to \sOm_V$ is the Grothendieck trace map.
\end{lemma}

\begin{proof}
We may assume that $V$ is a sufficiently small affine scheme.
Take a compactification of $g$ and $B$: a morphism $f \colon X \to Y$ and divisor $\barB$ on $X$ such that $U$ and $V$ are open sets in $X$ and $Y$ respectively, $f \restrict{U} = g$, and $\barB \restrict{U} = B$.
Applying \cite[Theorem 4.21]{HLS21} to $f$, we get the conclusion.
\end{proof}

\begin{lemma}\label{lem: regfincov}
Let $(A, \frakm)$ be a $d$-dimensional Noetherian normal local ring that has a dualizing complex. Suppose that $A$ has a regular finite cover $A \tomono B$. Then the following are equivalent.
\begin{enumerate}
\item $\mH^d_\frakm(A) \to \mH^d_\frakm\mleft(A^\abscl\mright)$ is injective.
\item $\mH^d_\frakm(A) \to \mH^d_\frakm(B)$ is injective.
\item $\Tr \colon \sOm_B \to \sOm_A$ is surjective.
\end{enumerate}
\end{lemma}

\begin{proof}
$\mH^d_\frakm(B) \to \mH^d_\frakm(B^\abscl) \isom \mH^d_\frakm(A^\abscl)$ is injective by the direct summand theorem. Thus (1) and (2) are equivalent. By Matlis duality, (2) and (3) are equivalent.
\end{proof}

The lemma below is similar to \cite[Theorem 7.2]{BMP+23}.

\begin{lemma}\label{lem: BMPAdj}
Suppose that $D$ is a normal prime divisor on $X$, $B \ge 0$ is a $\QQ$-Cartier divisor on $X$, and $D$ and $B$ have no common components. Let $\sL = \sO_X(L)$ be a line bundle on $X$ such that $L - B$ is semiample and big. Then the natural morphism $\sOm_X(D) \to \sOm_D$ induces a surjection
\[
    \mBMPAdj{D}(X, D+B, \sOm_X(D) \otimes \sL) \to \mBMP(D, B \restrict{D}, \sOm_D \otimes \sL)\text{.}
\]
\end{lemma}

\begin{proof}
Let $\sN \coloneqq \sO_{X^\abscl}(\nu^*(B - L))$.
Take $D^\abscl$ as in \myref{def: mBMP}: $D^\abscl \subset X^\abscl$ is a closed subscheme lying over $D \subset X$.
We have the following diagram with the exact rows
\begin{DIFnomarkup}
\begin{equation*}
    \begin{tikzcd}
        0 \arrow[r] & \sO_X(-D)\otimes\sL^\inv \arrow[r] \arrow[d]
        & \sL^\inv \arrow[r] \arrow[d]
        & \sO_D\otimes\sL^\inv \arrow[r] \arrow[d] & 0 \\
        0 \arrow[r] & \nu_*(\sO_{X^\abscl}(-D^\abscl)\otimes\sN) \arrow[r]
        & \nu_* \sN \arrow[r]
        & \nu_*(\sO_{D^\abscl}\otimes\sN) \arrow[r] & 0
    \end{tikzcd}_{\textstyle .}
\end{equation*}
\end{DIFnomarkup}We obtain the diagram
\begin{DIFnomarkup}
\begin{equation*}
    \begin{tikzcd}
        & \fHRRGa{d-1}{\frakm}\mleft(D,\sO_D\otimes\sL^\inv\mright)
        \arrow[r] \arrow[d,"\alpha"]
        & \fHRRGa{d}{\frakm}\mleft(X,\sO_X(-D)\otimes\sL^\inv\mright)
        \arrow[d,"\beta"] \\
        0
        \arrow[r]
        & \fHRRGa{d-1}{\frakm}\mleft(D^\abscl,\sO_{D^\abscl}\otimes\sN\mright)
        \arrow[r]
        & \fHRRGa{d}{\frakm}\mleft(X^\abscl,\sO_{X^\abscl}\mleft(-D^\abscl\mright)\otimes\sN\mright)
    \end{tikzcd}_{\textstyle ,}
\end{equation*}
\end{DIFnomarkup}where the bottom row is exact because $\fHRRGa{d-1}{\frakm}(X^\abscl,\sN) = 0$ by Bhatt's vanishing theorem (\myref{thm: Bhatt}). Hence, the map $\im \alpha \to \im \beta$ is injective. The Matlis dual of this map is the desired surjection by \cite[Lemma 4.8 (a), Definition 4.21, Lemma 4.24]{BMP+23}.
\end{proof}

\begin{corollary}[{cf.~\cite[Theorem 4.23]{HLS21}}]\label{cor: HLSAdj}
Suppose that $D$ is a normal prime divisor on $X$, $B \ge 0$ is a $\QQ$-Cartier divisor on $X$, and $D$ and $B$ have no common components. Then, the natural morphism $\sOm_X(D) \to \sOm_D$ induces a surjection
\[
    \sHLSAdjW{X}{D}{B} \toepi \sHLS(\sOm_D, B \restrict{D})\text{.}
\]
In particular, we have an inclusion
\[
    \im(\sHLS(\sOm_X, B-D) \tomono \sOm_X(D) \to \sOm_D) \supseteq \sHLS(\sOm_D, B \restrict{D})\text.
\]
\end{corollary}

Note that the latter inclusion can be regarded as a version of the usual restriction theorem for multiplier ideals \cite[Theorem 9.5.1]{LazII}.

\begin{proof}
Let $\sL$ be a sufficiently ample line bundle on $X$. Then $\mBMPAdj{D}(X, D+B, \sOm_X(D) \otimes \sL)$ (resp.~$\mBMP(D, B \restrict{D}, \sOm_D \otimes \sL)$) generates $\sHLSAdjW{X}{D}{B} \otimes \sL$ (resp.~$\sHLS(\sOm_D, B \restrict{D}) \otimes \sL$).
By \myref{lem: BMPAdj}, we obtain the desired surjection.
The latter inclusion follows since $\sHLSAdjW{X}{D}{B} \subset \sHLS(\sOm_X, B-D)$ by \cite[Lemma 4.8 (c)]{HLS21}.
\end{proof}

\begin{theorem}\label{thm: sing}
Let $X$ be a normal integral projective $R$-scheme. If $\sO_{X,x}$ satisfies \myref{cond: regseq} for every closed point $x \in X$, then $\sHLS(\sOm_X) = \sOm_X$ holds.
\end{theorem}

\begin{proof}
We prove $\sHLS(\sOm_X)_x = \sOm_{X,x}$ for each closed point $x \in X$. By assumption, there exist $r \ge 0$ and a regular sequence $a_1,\dots,a_r \in \sO_{X,x}$ such that $\sO_{X,x} / (a_1,\dots,a_r)$ satisfies \myref{cond: fincov}. We use induction on $r \ge 0$.

Suppose that $r = 0$. Since $\sO_{X,x}$ satisfies \myref{cond: fincov}, by \myref{lem: regfincov}, we have a neighborhood $U$ of $x$ and a regular finite cover $g \colon U' \to U$ such that $\Tr \colon g_*\sOm_{U'} \to \sOm_U$ is surjective. By \myref{lem: B0-fin}, we obtain $\Tr (g_*\sHLS(\sOm_{U'})) = \sHLS(\sOm_U)$. Now, we have $\sHLS(\sOm_{U'}) = \sOm_{U'}$ by \cite[Proposition 4.24]{HLS21}. Therefore, we deduce that
\[
    \sHLS(\sOm_U) = \Tr (g_*\sHLS(\sOm_{U'})) = \Tr (g_*\sOm_{U'}) = \sOm_U\text{.}
\]
Thus we obtain $\sHLS(\sOm_{X})_{x} = \sOm_{X,x}$.

Suppose that $r \ge 1$. Let $U \subset X$ be a sufficiently small affine neighborhood of $x$. Then we see that $D \coloneqq \divisor_U a_1$ is a normal prime divisor, where we recall that $a_1, \dotsc, a_r$ is the regular sequence of \myref{cond: regseq}. Set $\barD$ to be the closure of $D$ in $X$. Replacing $X$ with a blowup of it, we may assume that $\barD$ is also a normal prime divisor. Indeed, the normalization morphism $\scNor\barD \to \barD$ is the blowup along a subscheme $Z$ of $\barD$ by \cite[Theorem 8.1.24]{Liu02}, and the blowup of $X$ along $Z$ is the desired replacement. We obtain
\[
    \im\mleft(\sHLS\mleft(\sOm_X, -\barD\mright) \to \sOm_{\barD}\mright) \supseteq \sHLS\mleft(\sOm_{\barD}\mright)
\]
from \myref{cor: HLSAdj}.
Therefore, since $\sHLS(\sOm_{D})_x = \sOm_{D,x}$ by the induction hypothesis, we deduce that
\[
    \im(\sHLS(\sOm_U) \otimes \sO_U(D) \to \sOm_D)_x = \im(\sHLS(\sOm_U, -D) \to \sOm_{D})_x = \sOm_{D,x}\text.
\]
Now, since $D$ is a normal Cartier divisor, we have
\[
    (\sOm_U(D) \otimes \sO_D)^\reflHull \isom \sOm_D\text{.}
\]
Since $\sOm_U$ is \serreS{2}, the module $\sOm_U(D) \otimes \sO_D$ is \serreS{1}, and hence the morphism $\sOm_U(D) \otimes \sO_D \to \sOm_D$ is injective. We thus get the diagram
\begin{DIFnomarkup}
\begin{equation*}
    \begin{tikzcd}[row sep=small]
        (\sHLS(\sOm_{U})\otimes\sO_U(D))_x \arrow[rr, two heads, bend left] \arrow[r, "\alpha"] & (\sOm_{U}(D)\otimes\sO_D)_x \arrow[r, hook] & \sOm_{D,x}
    \end{tikzcd}\text{,}
\end{equation*}
\end{DIFnomarkup}and hence see that $\alpha$ is surjective. We conclude from Nakayama's lemma that $\sHLS(\sOm_U)_x = \sOm_{U,x}$.
\end{proof}

\begin{theorem}\label{thm: gg-alt}
We work in \myref{set: RelB0}.
Assume that $f$ is generically finite. Let $\sL$ be a globally generated ample line bundle on $Y$. Define $V$ to be the maximal open subset of $Y$ satisfying $\relBWo{f}{X} \restrict{V} = f_*\sHLS(\sOm_X) \restrict{V}$. Suppose that $\sO_{X,x}$ satisfies \myref{cond: regseq} for each closed point $x \in X$. Then, $V$ contains the locus over which $f$ is finite, and
\[
    f_*\sOm_X \otimes \sL^l
\]
is globally generated over $V$ for $l \ge n+1$.
\end{theorem}

\begin{proof}
Combine \myref{thm: sing} and \myref{cor: gg-alt-tau}.
\end{proof}

\section{Global generation for pluricanonical bundles}
The purpose of this section is to prove \myref{main: main}.
We still maintain \myref{set: RelB0}.

We need the following variant of \myref{prop: RelB0-Keeler}.
\begin{proposition}\label{prop: RelB0-KeelerOp}
Let $\sM = \sO_X(M)$ be a line bundle on $X$. Suppose that $V \subset Y$ is an open subset, $U \coloneqq f^\inv(V)$, and $\sM \restrict{U}$ is $f \restrict{U}$-ample. Let $t \in \QQ$, and set $B \coloneqq tM$. Then for sufficiently large $m$,
\[
    \relBWM{f}{X}{B}{\sM^m} \restrict{V} = f_*\mleft(\sHLS\mleft(\sOm_X, B\mright) \otimes \sM^m\mright) \restrict{V}\text{.}
\]
\end{proposition}

\begin{proof}
We may assume that $V \neq \emptyset$.
For $j \gg 0$, the divisor $jM \restrict{U}$ is $f \restrict{U}$-free. By replacing $M$ by $jM$, we can assume that $j = 1$: more precisely, we reduce the assertion for $(M, t)$ to the assertions for $(jM, t/j), (jM, (t+1)/j), \dotsc, (jM, (t+j-1)/j)$. Let $\mu \colon X' \to X$ be the normalized blowup along the relative base locus $\relBs{f}(M)$; then $\mu^*M = M' + F'$, where $F'$ is the (relatively) fixed part. Define $f' \coloneqq f \circ \mu \colon X' \to Y$, $X'' \coloneqq \relProj_X \bigoplus_{l \ge 0} f'_* \sO_{X'}\mleft(lM'\mright)$, and $\sO_{X''}\mleft(M''\mright) = \sM'' \coloneqq \sO_{X''}(1)$. Let
\begin{DIFnomarkup}
\[
    \begin{tikzcd}
        X' \arrow[r, "\mu''"] & X'' \arrow[r, "f''"] & Y
    \end{tikzcd}
\]
\end{DIFnomarkup}be the natural morphisms. Then $\mu''^*M''=M'$. By \myref{prop: RelB0-Keeler} and the $f''$-ampleness of $\sM''$, we have
\begin{equation}\label{eq: RelB0-KeelerOp}
    \relBWM{f''}{X''}{tM''}{\mleft(\sM''\mright)^m} = f''_*\mleft(\sHLS\mleft(\sOm_{X''}, tM''\mright) \otimes \mleft(\sM''\mright)^m\mright)
\end{equation}
for sufficiently large $m$.
Note that $\mu$ and $\mu''$ are isomorphic over $V$.
Write $U'' \coloneqq (f'')^\inv(V)$.
Thus we see that
\begin{align*}
    &\mathrel{\phantom{=}} \relBWM{f}{X}{B}{\sM^m} \restrict{V} \\
    &= \relBWM{f'}{X'}{\mu^*B}{\mu^*\sM^m} \restrict{V} &&\text{by \myref{prop: RelB0-fmu}} \\
    &= \restrictx{\relBWM{f'}{X'}{tM'}{\mleft(\sM'\mright)^m}}{V} &&\text{by \myref{lem: RelB0-Op}} \\
    &= \restrictx{\relBWM{f''}{X''}{tM''}{\mleft(\sM''\mright)^m}}{V} &&\text{by \myref{prop: RelB0-fmu}} \\
    &= \restrictx{f''_*\mleft(\sHLS\mleft(\sOm_{X''}, tM''\mright) \otimes \mleft(\sM''\mright)^m\mright)}{V} &&\text{by \myref{eq: RelB0-KeelerOp}} \\
    &= \mleft(\restrictx{f''}{U''}\mright)_*\mleft(\sHLS\mleft(\sOm_{U''}, \restrictx{tM''}{U''}\mright) \otimes \restrictx{\mleft(\sM''\mright)^m}{U''}\mright) \\
    &= f_*(\sHLS(\sOm_X, B) \otimes \sM^m) \restrict{V}\text.
\end{align*}
\end{proof}

To show \myref{main: main}, we use \myref{thm: RelGgOp} in the following situation.
\begin{remark}\label{rmk: BigSA}
Let $B$ be a $\QQ$-Cartier divisor on $X$ and let $\sL = \sO_Y(L)$ be an ample line bundle on $Y$. Suppose that $V \subset Y$ is a dense open set, $U \coloneqq f^\inv(V)$, $-B \restrict{U}$ is $f \restrict{U}$-ample, and $-B-\delta f^*L$ is semiample over $U$ for some $\delta > 0$.

Then $\BBp(-B) \cap U = \emptyset$. In particular, $-B$ is semiample over $U$ and big.
\end{remark}

We now prove \myref{main: main}.

\begin{theorem}\label{thm: RelFuj}
Suppose that $X$, $Y$, $f$, and $n$ are as in \myref{set: RelB0}. Assume that $\Delta$ is an effective $\QQ$-divisor on $X$ and $\dK_X + \Delta$ is $\QQ$-Cartier with index $r \ge 1$. Let $V$ be an open set in $Y$ and $U \coloneqq f^\inv(V)$. Let $j \ge 1$ and let $\sL$ be an ample line bundle on $Y$ such that $\sL^j$ is globally generated.
Assume
\begin{enumerate}
\item \label{enum: relFujAmp}$(\dK_X + \Delta) \restrict{U}$ is $f \restrict{U}$-ample.
\item \label{enum: relFujTau}$\sHLS(\sO_X, \Delta) \restrict{U} = \sO_U$. (Recall that the left-hand side is the $\bmplus$-test ideal by \cite{HLS21}: see \myref{def: sHLS}.)
\end{enumerate}

Then, there exists $m_0 \ge 1$ such that
\[
    f_* \sO_X(m(\dK_X + \Delta)) \otimes \sL^l
\]
is globally generated over $V$ for every $m \ge m_0$ divisible by $r$ and $l \ge m(jn+1)$.
\end{theorem}

\begin{proof}
We may assume $V \neq \emptyset$.
From \eqref{enum: relFujAmp} and \myref{prop: RelB0-KeelerOp}, we see that
\[
    \relBOM{f}{X}{\Delta}{\sO_X(m(\dK_X+\Delta))} \to f_*(\sHLS(\sO_X, \Delta) \otimes \sO_X(m(\dK_X+\Delta)))
\]
is isomorphic over $V$ for $m \gg 0$ with $r \divides m$.
Therefore, by \eqref{enum: relFujAmp} and \eqref{enum: relFujTau}, there exists $m_0 \gg 0$ such that $m\mleft(\dK_X+\Delta\mright) \restrict{U}$ is $f \restrict{U}$-free and
\[
    \relBOM{f}{X}{\Delta}{\sO_X(m(\dK_X+\Delta))} \to f_*\sO_X(m(\dK_X+\Delta))
\]
is isomorphic over $V$ for every $m \ge m_0$ with $r \divides m$.

Take a divisor $L$ such that $\sL = \sO_Y(L)$. Set
\[
    \epsilon \coloneqq \inf \setIn{t \in \QQge{0}}{\text{$\dK_X + \Delta + tf^*L$ is semiample over $U$}}\text,
\]
which is finite by \eqref{enum: relFujAmp}.

\begin{claim}\label{claim: RelFuj}
For an integer $m \ge m_0$ with $r \divides m$ and integer $s > (m-1) \epsilon$, 
\[
    f_* \sO_X(m(\dK_X + \Delta)) \otimes \sL^{s + jn}
\]
is globally generated over $V$.
\end{claim}

We show \myref{claim: RelFuj}. From the definition of $\epsilon$ and \myref{rmk: BigSA}, we deduce that $(m-1)(\dK_X + \Delta) + sf^*L$ is semiample over $U$ and big. Hence we see from \myref{thm: RelGgOp} that
\[
    \relBOM{f}{X}{\Delta}{\sO_X(m(\dK_X+\Delta))} \otimes \sL^{s+jn}
\]
is globally generated over $V$. Combined with the definition of $m_0$, this completes the proof of \myref{claim: RelFuj}.

Let $m \ge m_0$ with $r \divides m$. Applying \myref{claim: RelFuj} to $s = s_m \coloneqq \floor{(m-1)\epsilon} + 1$, we see that $f_* \sO_X(m(\dK_X + \Delta)) \otimes \sL^{s_m + jn}$ is globally generated over $V$. Hence
\[
    f^*f_*\sO_X(m(\dK_X + \Delta)) \otimes f^*\sL^{s_m + jn}
\]
is globally generated over $U$, and so is
\[
    \sO_X(m(\dK_X + \Delta)) \otimes f^*\sL^{s_m + jn}
\]
since $m(\dK_X+\Delta) \restrict{U}$ is $f \restrict{U}$-free. From the definition of $\epsilon$, it follows that
\[
    \epsilon \le \frac{s_m + jn}{m}\text{,}
\]
which gives
\[
    \epsilon \le jn+1
\]
since $s_m = \floor{(m-1)\epsilon} + 1 \le (m-1)\epsilon + 1$.

We now apply \myref{claim: RelFuj} to $s = (m-1)(jn+1) + 1$, completing the proof.
\end{proof}

\begin{remark}\label{rmk: RelFujSing}
Suppose that $(R, \maxmR)$ is a complete DVR of mixed characteristic and $X \to \Spec R$ is flat.
In this case, by \myref{rmk: pPlusRegular}, the assumption \eqref{enum: relFujTau} in \myref{thm: RelFuj} can be replaced by the following condition: for every $x \in U$, $(\completion{\sO_{X,x}}, \Delta \restrict{\completion{\sO_{X,x}}})$ is $\text{BCM}$-regular if the residue characteristic of $\sO_{X,x}$ is $p$, and is KLT if it is $0$.
Therefore we get \myref{main: main}.
\end{remark}

We give a quick application of \myref{thm: RelFuj}.

\begin{proposition}
Suppose that $X$ and $Y$ are regular.
Assume that $\sN$ is a line bundle on $Y$ and $\sOm_X$ is numerically equivalent to $f^*\sN$. Let $\sL$ be a globally generated ample line bundle on $Y$. Then $\sN \otimes \sL^{n+1}$ is nef.
\end{proposition}

\begin{proof}
Take divisors $N$ and $L$ such that $\sN = \sO_Y(N)$ and $\sL = \sO_Y(L)$.
By \cite[Theorem 2.17]{BMP+23}, there is an ample effective divisor $A$ on $X$ that is regular.
Let $l > 1$ and $\Delta \coloneqq A/l$.
By \cite[Proposition 4.24]{HLS21}, we have $\sHLS(\sO_X, \Delta) = \sO_X$.
Since $\dK_X \numeq f^*N$, we see that $\dK_X+\Delta$ is $f$-ample.
Therefore by \myref{thm: RelFuj},
\[
    f_*\sO_X(m(\dK_X+\Delta)) \otimes \sL^{m(n+1)}
\]
is globally generated for all $m \gg 0$ with $l \divides m$.
Since $m(\dK_X+\Delta)$ is $f$-free for $m \gg 0$, it follows that
\[
    \sO_X(m(\dK_X+\Delta)) \otimes f^*\sL^{m(n+1)}
\]
is globally generated.
Hence $\dK_X+\Delta+(n+1)f^*L$ is nef, and so is $f^*N+A/l+(n+1)f^*L$. Letting $l \to \infty$, we deduce that $f^*(N+(n+1)L)$ is nef, and therefore $\sN \otimes \sL^{n+1}$ is nef.
\end{proof}

\section{Weak positivity}
In this section, we show \myref{main: wp}.
The proof is based on Viehweg's fiber product trick.
First, we give a lemma on relatively SNC divisors after recalling the definition.

\begin{definition}\label{def: rel-snc}
Let $g \colon U \to V$ be a dominant smooth morphism of relative dimension $d$ between regular integral quasi-projective $R$-schemes.
Suppose that $D$ is a reduced divisor on $U$ and $D = \sum_{j \in J} D_j$ is the decomposition into irreducible components.
We say that $D$ is \emph{relatively SNC over $V$} if $\cap_{j \in J'} D_j \tomono U \to V$ is a smooth morphism of relative dimension $d - \abs{J'}$ for every nonempty subset $J' \subset J$.
\end{definition}

Note that if $D$ is relatively SNC over $V$, then $D$ is an SNC divisor on $U$ by \citestacks{0BIA}.

\begin{lemma}\label{lem: sncprod}
Suppose that $g_i \colon U_i \to V$ is a dominant smooth morphism of relative dimension $d_i$ between regular integral quasi-projective $R$-schemes for $i = 1,\dots,l$. Let $\Delta_i$ be a $\QQ$-divisor on $U_i$ such that $\floor{\Delta_i} = 0$ and $\Supp \Delta_i$ is relatively SNC over $V$ for $i = 1,\dots,l$.

Set $U^{(l)} \coloneqq U_1 \times_V \dots \times_V U_l$. Write $\pr_i \colon U^{(l)} \to U_i$ for the $i$th projection. Define $\Delta^{(l)} \coloneqq \sum_{i=1}^l \pr_i^* \Delta_i$. Then $\floor{\Delta^{(l)}} = 0$ and $\Supp \Delta^{(l)}$ is relatively SNC over $V$. In particular, $\Supp \Delta^{(l)}$ is an SNC divisor on $U^{(l)}$.
\end{lemma}

\begin{proof}
By induction, we can assume that $l = 2$.
Note that for a prime divisor $D_1$ on $U_1$, $\pr_1^* D_1 = D_1 \times_V U_2$ is smooth over $V$ and the irreducible components of $\pr_1^* D_1$ are disjoint.
Let $\Supp \Delta_i = \sum_{j \in J_i} D_{ij}$ be the decomposition into irreducible components for $i = 1$, $2$. Then we have $\Supp \Delta^{(l)} = \sum_{j \in J_1} \pr_1^* D_{1j} + \sum_{j \in J_2} \pr_2^* D_{2j}$. We also obtain $\floor{\Delta^{(l)}} = 0$.
For subsets $J'_1 \subset J_1$, $J'_2 \subset J_2$, by assumption,
\[
    \bigcap_{j \in J'_1} \pr_1^* D_{1j} \cap \bigcap_{j \in J'_2} \pr_2^* D_{2j}
    = \mleft(\bigcap_{j \in J'_1} D_{1j} \mright) \times_V \mleft(\bigcap_{j \in J'_2} D_{2j} \mright)
\]
is smooth over $V$ of relative dimension $d_1+d_2 - \abs{J'_1} - \abs{J'_2}$.
Therefore, $\Supp \Delta^{(l)}$ is relatively SNC over $V$.
\end{proof}

We proceed to prove \myref{main: wp}.

\begin{setting}\label{set: wp}
Let $f \colon X \to Y$ be a smooth surjection between regular integral projective $R$-schemes. Define $n \coloneqq \dim Y_{\maxmR}$.

Let $\Delta$ be an effective $\QQ$-divisor on $X$ such that $\dK_X + \Delta$ is $\QQ$-Cartier with index~$r$. Suppose that $V \subset Y$ is an open set, $U \coloneqq f^\inv(V)$, and $(\dK_X + \Delta) \restrict{U}$ is $f \restrict{U}$-ample. Assume that $\floor{\Delta \restrict{U}} = 0$ and $\Supp \Delta\restrict{U}$ is relatively SNC over $V$.
\end{setting}

\begin{theorem}\label{thm: Wp}
In \myref{set: wp}, $\dK_{X/Y} + \Delta$ is weakly positive over $U$.
\end{theorem}

\begin{proof}
Let $\sL = \sO_Y(L)$ be a globally generated ample line bundle on $Y$. Fix $l \ge 1$. Let
\[
    X^{(l)} \coloneqq X \times_Y \dots \times_Y X, \quad \Delta^{(l)} \coloneqq \sum_{i=1}^l \pr_i^* \Delta\text{,}
\]
and let $f^{(l)} \colon X^{(l)} \to Y$ be the natural morphism. By \myref{lem: sncprod}, $\floor*{\restrictx{\Delta^{(l)}}{U^{(l)}}} = 0$ and $\Supp \restrictx{\Delta^{(l)}}{U^{(l)}}$ is SNC. Hence we obtain $\restrictx{\sHLS\mleft(\sO_{X^{(l)}}, \Delta^{(l)}\mright)}{U^{(l)}} = \restrictx{\sO_{X^{(l)}}}{U^{(l)}}$ by \cite[Proposition 4.24]{HLS21}. We have
\[
    \dK_{X^{(l)}/Y} + \Delta^{(l)} = \sum_{i=1}^l \pr_i^*(\dK_{X/Y} + \Delta)\text.
\]
We thus see that $\restrictx{\mleft(\dK_{X^{(l)}} + \Delta^{(l)}\mright)}{U^{(l)}}$ is $\restrictx{f^{(l)}}{U^{(l)}}$-ample. It follows from \myref{thm: RelFuj} that
\[
    \mleft(f^{(l)}\mright)_* \sO_{X^{(l)}}\mleft(m\mleft(\dK_{X^{(l)}} + \Delta^{(l)}\mright)\mright) \otimes \sL^{m(n+1)}
\]
is globally generated over $V$ for every $m \gg 0$ with $r \divides m$.

Let $m \gg 0$ with $r \divides m$.
Since $(\dK_{X/Y} + \Delta) \restrict{U}$ is $f \restrict{U}$-ample, we see that $m(\dK_{X/Y} + \Delta) \restrict{U}$ is $f \restrict{U}$-free.
Since $\restrictx{\mleft(\dK_{X^{(l)}} + \Delta^{(l)}\mright)}{U^{(l)}}$ is $\restrictx{f^{(l)}}{U^{(l)}}$-ample, the multiplication map
\[
    \mleft(f^{(l)}\mright)_*\sO_{X^{(l)}}\mleft(m\sum_{i=1}^l \pr_i^*(\dK_{X/Y} + \Delta)\mright) \to f_*\sO_X(lm(\dK_{X/Y} + \Delta))
\]
is surjective over $V$ (see \cite[Example 1.2.22]{LazI}). Thus it follows that
\[
    f_* \sO_X(lm(\dK_{X/Y} + \Delta)) \otimes \sOm_Y^m \otimes \sL^{m(n+1)}
\]
is globally generated over $V$.
By the $f \restrict{U}$-freeness of $m(\dK_{X/Y} + \Delta)$,
\[
    lm(\dK_{X/Y} + \Delta) + mf^*(\dK_Y + (n+1)L)
\]
is globally generated over $U$, and hence 
\[
    \dK_{X/Y} + \Delta + \frac{1}{l}f^*\mleft(\dK_Y + (n+1)L\mright)
\]
is semiample over $U$. This shows that $\dK_{X/Y} + \Delta$ is weakly positive over $U$. Indeed, taking an ample divisor $A$ on $X$ such that $A - f^*(\dK_Y + (n+1)L)$ is semiample, we see that
\[
    \dK_{X/Y} + \Delta + \frac{1}{l}A
\]
is semiample over $U$ for all $l \ge 1$, as desired.
\end{proof}

As a corollary, we obtain a global generation result similar to \cite[Theorem 1.7]{Fujn23} and \cite[Theorem 1.9]{Ejir24}.

\begin{proposition}\label{prop: Fujino}
We work in \myref{set: wp}. Let $\sL$ be a globally generated ample line bundle on $Y$. Then
\[
    f_*\sO_X(m(\dK_{X/Y} + \Delta)) \otimes \sOm_Y \otimes \sL^{n+1}
\]
is globally generated over $V$ for all $m \gg 0$ with $r \divides m$.
\end{proposition}

\begin{proof}
Take a divisor $L$ such that $\sL = \sO_Y(L)$.
Let $m \gg 0$ with  $r \divides m$. By \myref{thm: Wp}, $\dK_{X/Y} + \Delta$ is weakly positive over $U$. Since it is also $f \restrict{U}$-ample, $(m-1)(\dK_{X/Y} + \Delta) + f^*L$ is semiample over $U$ and big by \myref{rmk: BigSA}. From \myref{thm: RelGgOp}, it follows that
\[
    \relBOM{f}{X}{\Delta}{\sO_X(m(\dK_{X/Y}+\Delta))} \otimes \sOm_Y \otimes \sL^{n+1}
\]
is globally generated over $V$. The assertion now follows from \myref{prop: RelB0-KeelerOp} and \cite[Proposition 4.24]{HLS21}.
\end{proof}

We give a result related to the Iitaka conjecture.

\begin{proposition}\label{prop: Iitaka}
In \myref{set: wp},
\[
    f(\BBp(\dK_X+\Delta)) \cap V \subset \BBp(\dK_Y) \cap V\text{.}
\]
\end{proposition}

\begin{proof}
Replacing $V$ with $V \setminus \BBp(\dK_Y)$, we may assume that $\BBp(\dK_Y) \cap V = \emptyset$.
By \myref{thm: Wp}, we have
\[
    \BBm(\dK_{X/Y}+\Delta) \cap U = \emptyset\text{.}
\]
Since $(\dK_X+\Delta) \restrict{U}$ is $f\restrict{U}$-ample and $\BBp(\dK_Y) \cap V = \emptyset$, we see that for $l \gg 0$,
\[
    \BBp\mleft(\frac{1}{l}(\dK_X+\Delta) + f^*\dK_Y\mright) \cap U = \emptyset\text.
\]
Therefore we deduce that
\[
    \BBp\mleft((\dK_{X/Y}+\Delta) + \mleft(\frac{1}{l}(\dK_X+\Delta) + f^*\dK_Y\mright)\mright) \cap U = \emptyset\text.
\]
The left-hand side equals $\BBp(\dK_X+\Delta) \cap U$, and the assertion follows.
\end{proof}

\begin{remark}
By \myref{prop: Iitaka}, if $\dK_Y$ is big, then $\dK_X+\Delta$ is big.
This is regarded as a special case of the Iitaka conjecture in mixed characteristic.

However, this assertion can readily be seen from the analogous result in characteristic zero (see, for example, \cite[Theorem 1.2]{KoPa17}) since bigness can be checked on the generic fiber over $R$.
The interest of \myref{prop: Iitaka} is that it applies also to points of positive characteristic.
\end{remark}

We prove that images of Fano schemes under smooth morphisms are again Fano.

\begin{proposition}\label{prop: Fano}
Suppose that $f \colon X \to Y$ is a smooth surjection between regular integral projective $R$-schemes. Let $N$ be a $\QQ$-divisor on $Y$. Then we have the following.

\begin{enumerate}
\item \label{enum: fano-nef}If $-\dK_X - f^*N$ is nef, so is $-\dK_Y - N$.
\item \label{enum: fano-fano}If $-\dK_X - f^*N$ is ample, so is $-\dK_Y - N$.
\end{enumerate}
\end{proposition}

We present two proofs of \myref{prop: Fano}: one is based on \myref{thm: Wp}, and the other relies on a result in positive characteristic.
First, we give a lemma for the former proof.

\begin{lemma}\label{lem: Bertini-mixed}
Let $f \colon X \to Y$ be a surjection between normal integral projective $R$-schemes. Suppose that $f$ is smooth over a neighborhood of a closed point $y \in Y$. Let $\sA$ be an ample line bundle on $X$.

Then for every $m \gg 0$, there is an effective divisor $D \in \abs*{\sA^m}$ such that $\Supp D$ contains no components of $\fiber{X}{f}{y}$ and $D \tomono X \to Y$ is smooth over a neighborhood of $y$.
\end{lemma}

\begin{proof}
Let $m \gg 0$. We see that the map
\[
    \alpha \colon \mH^0(X, \sA^m) \to \mH^0\mleft(\fiber{X}{f}{y}, \sA^m \restrict{\fiber{X}{f}{y}}\mright)
\]
is surjective.

If the residue field $\rkap(y)$ of $y$ is infinite, then by the Bertini theorem, we obtain a divisor $D_y \in \abs*{\sA^m \restrict{\fiber{X}{f}{y}}}$ that is smooth over $\rkap(y)$. If the residue field $\rkap(y)$ is finite, we can take such $D_y$ by Poonen's Bertini theorem \cite{Poon04} and $m \gg 0$.

By the surjectivity of $\alpha$, we get a divisor $D$ satisfying $D \restrict{\fiber{X}{f}{y}} = D_y$, which is the desired divisor, since smoothness is an open condition.
\end{proof}    

\begin{proof}[Proof of \myref{prop: Fano}]
We show that \eqref{enum: fano-nef} implies \eqref{enum: fano-fano}. Assume that $-\dK_X - f^*N$ is ample. Take an ample $\QQ$-divisor $L$ on $Y$ such that $-\dK_X - f^*N - f^*L$ is ample. It follows from \eqref{enum: fano-nef} that $-\dK_Y - N - L$ is nef, and consequently that $-\dK_Y - N$ is ample.

It only remains to show \eqref{enum: fano-nef}. Assume that $-\dK_X - f^*N$ is nef. Suppose by contradiction that $-\dK_Y - N$ is not nef. Then $f^*(-\dK_Y - N)$ is not nef, and hence there exists a closed point $y \in f(\BBm(f^*(-\dK_Y - N)))$. Take an ample divisor $A$ on $X$. For some $0 < \epsilon \ll 1$, we have
\begin{equation}\label{eq: fano-pf}
    \begin{split}
        y &\in f(\BBm(f^*(-\dK_Y - N) + \epsilon A)) \\
        &= f(\BBm(\dK_{X/Y} + (-\dK_X - f^*N + \epsilon A)))\text.
    \end{split}
\end{equation}

Since $-\dK_X - f^*N$ is nef, $-\dK_X - f^*N + \epsilon A$ is ample. From \myref{lem: Bertini-mixed}, it follows that for every $m \gg 0$, there exists a divisor $D$ on $X$ such that
\[
    D \lineq m(-\dK_X - f^*N + \epsilon A)
\]
and $D$ is smooth over a neighborhood $V$ of $y$. Set $\Delta \coloneqq D/m$.
We see that the divisor $\dK_X + \Delta  \Qlineq -f^*N + \epsilon A$ is $f$-ample. Applying \myref{thm: Wp}, we obtain
\[
    f(\BBm(\dK_{X/Y} + \Delta)) \cap V = \emptyset\text,
\]
a contradiction to \myref{eq: fano-pf}.
\end{proof}

\begin{proof}[Another proof of \myref{prop: Fano}]
We prove \myref{prop: Fano} \eqref{enum: fano-nef} by using a similar result in positive characteristic (see, for example, \cite[Corollary 3.15]{Debarre}, \cite{Ejir19}).
Note that \eqref{enum: fano-nef} implies \eqref{enum: fano-fano} as explained above.

It is enough to show that $(-\dK_Y - N) \restrict{C_0}$ is nef for every projective curve $C_0 \subset Y$ over $k \coloneqq R / \maxmR$.
Let $C_1 \coloneqq (C_0)_{\algcl{k}}$ be the base change of $C_0$ to $\algcl{k}$ and $C_2 \coloneqq \scNor{C_1}$ its normalization.
Because nefness can be checked on $C_2$, it is sufficient to prove that $g^*(-\dK_Y - N)$ is nef for every nonconstant morphism $g \colon C \to Y$ from a smooth projective curve $C$ over $\algcl{k}$.
Form the following fiber product diagram:
\begin{DIFnomarkup}
\begin{equation*}
    \begin{tikzcd}
        X_C \ar[r, "g'"] \ar[d, "f'"] & X \ar[d, "f"] \\
        C \ar[r, "g"] & Y
    \end{tikzcd}_{\textstyle.}
\end{equation*}
\end{DIFnomarkup}Note that since $f$ is smooth, $f'$ is also smooth.
Hence, by $(g')^*(\Omega_{X/Y}) \isom \Omega_{X_C/C}$, we have $(g')^*(\dK_{X/Y}) = \dK_{X_C/C}$ (cf.~\cite[Theorem 3.6.1]{Conrad00}). Hence $-(g')^*\dK_X + (f')^*g^*\dK_Y = -\dK_{X_C} + (f')^*\dK_C$.
Since $-\dK_X - f^*N$ is nef by assumption, so is
\[
    (g')^*(-\dK_X - f^*N) = -\dK_{X_C} + (f')^*D\text,
\]
where we set
\[
    D \coloneqq \dK_C + g^*(-\dK_Y - N)\text.
\]
Hence we deduce that $-\dK_C + D = g^*(-\dK_Y - N)$ is also nef by \cite[Corollary 3.15]{Debarre}.
Thus, $-\dK_Y - N$ is nef.
\end{proof}

\begin{corollary}
Suppose that $f \colon X \to Y$ is a smooth surjection between regular integral projective $R$-schemes.
If $-\dK_X$ is nef and big, so is $-\dK_Y$.
\end{corollary}

\begin{proof}
Assume that $-\dK_X$ is nef and big. Then $-\dK_Y$ is nef by \myref{prop: Fano}. Since bigness can be checked on the generic fiber over $R$, it follows from \cite[Theorem 1.1]{FuGo12} that $-\dK_Y$ is big.
\end{proof}

%% file: main.bbl
\newcommand{\etalchar}[1]{$^{#1}$}
\providecommand{\bysame}{\leavevmode\hbox to3em{\hrulefill}\thinspace}
\providecommand{\MR}{\relax\ifhmode\unskip\space\fi MR }
\providecommand{\MRhref}[2]{%
  \href{http://www.ams.org/mathscinet-getitem?mr=#1}{#2}
}
\providecommand{\href}[2]{#2}
\begin{thebibliography}{BMP{\etalchar{+}}25}

\bibitem[And18]{Andr18}
Yves Andr\'{e}, \emph{La conjecture du facteur direct}, Publ. Math. Inst. Hautes \'{E}tudes Sci. \textbf{127} (2018), 71--93. \MR{3814651}

\bibitem[Bha21]{Bhat20}
Bhargav Bhatt, \emph{Cohen-macaulayness of absolute integral closures}, \url{https://arxiv.org/abs/2008.08070v2}, 2021.

\bibitem[BKK{\etalchar{+}}15]{BKKMSU15}
Thomas Bauer, S\'andor~J. Kov\'acs, Alex K\"uronya, Ernesto~C. Mistretta, Tomasz Szemberg, and Stefano Urbinati, \emph{On positivity and base loci of vector bundles}, Eur. J. Math. \textbf{1} (2015), no.~2, 229--249. \MR{3386236}

\bibitem[BMP{\etalchar{+}}23]{BMP+23}
Bhargav Bhatt, Linquan Ma, Zsolt Patakfalvi, Karl Schwede, Kevin Tucker, Joe Waldron, and Jakub Witaszek, \emph{Globally $\pmb{+}$-regular varieties and the minimal model program for threefolds in mixed characteristic}, Publ. Math. Inst. Hautes {\'E}tudes Sci. \textbf{138} (2023), no.~1, 69--227.

\bibitem[BMP{\etalchar{+}}25]{BMP+24}
Bhargav Bhatt, Linquan Ma, Zsolt Patakfalvi, Karl Schwede, Kevin Tucker, Joe Waldron, Jakub Witaszek, and Rankeya Datta, \emph{Test ideals in mixed characteristic: a unified theory up to perturbation}, 2025.

\bibitem[Con00]{Conrad00}
Brian Conrad, \emph{Grothendieck duality and base change}, Lecture Notes in Mathematics, vol. 1750, Springer-Verlag, Berlin, 2000. \MR{1804902}

\bibitem[Deb01]{Debarre}
Olivier Debarre, \emph{Higher-dimensional algebraic geometry}, Universitext, Springer-Verlag, New York, 2001. \MR{1841091}

\bibitem[DM19]{DuMu19}
Yajnaseni Dutta and Takumi Murayama, \emph{Effective generation and twisted weak positivity of direct images}, Algebra Number Theory \textbf{13} (2019), no.~2, 425--454. \MR{3927051}

\bibitem[Eji17]{Ejir17}
Sho Ejiri, \emph{Weak positivity theorem and {F}robenius stable canonical rings of geometric generic fibers}, J. Algebraic Geom. \textbf{26} (2017), no.~4, 691--734. \MR{3683424}

\bibitem[Eji19]{Ejir19}
\bysame, \emph{Positivity of anticanonical divisors and {$F$}-purity of fibers}, Algebra Number Theory \textbf{13} (2019), no.~9, 2057--2080. \MR{4039496}

\bibitem[Eji23a]{Ejir23b}
\bysame, \emph{Notes on direct images of pluricanonical bundles}, Eur. J. Math. \textbf{9} (2023), no.~4, Paper No. 96, 9. \MR{4652922}

\bibitem[Eji23b]{Ejir23a}
\bysame, \emph{Notes on {F}robenius stable direct images}, J. Algebra \textbf{633} (2023), 464--473. \MR{4617999}

\bibitem[Eji24]{Ejir24}
\bysame, \emph{Direct images of pluricanonical bundles and {F}robenius stable canonical rings of fibers}, Algebr. Geom. \textbf{11} (2024), no.~1, 71--110. \MR{4680014}

\bibitem[FG12]{FuGo12}
Osamu Fujino and Yoshinori Gongyo, \emph{On images of weak {F}ano manifolds}, Math. Z. \textbf{270} (2012), no.~1-2, 531--544. \MR{2875847}

\bibitem[FG14]{FuGo14}
\bysame, \emph{On images of weak {F}ano manifolds {II}}, Algebraic and complex geometry, Springer Proc. Math. Stat., vol.~71, Springer, Cham, 2014, pp.~201--207. \MR{3278574}

\bibitem[Fuj23]{Fujn23}
Osamu Fujino, \emph{On mixed-$\omega$-sheaves}, 2023, \url{https://arxiv.org/abs/1908.00171v3}.

\bibitem[Gro61]{EGAIII}
A.~Grothendieck, \emph{\'{E}l\'{e}ments de g\'{e}om\'{e}trie alg\'{e}brique. {III}. \'{E}tude cohomologique des faisceaux coh\'{e}rents. {I}.}, Inst. Hautes \'{E}tudes Sci. Publ. Math. (1961), no.~11, 167. \MR{217085}

\bibitem[Har03]{HarN:2003}
Nobuo Hara, \emph{A characteristic {$p$} analog of multiplier ideals and its applications}, \url{http://hdl.handle.net/2433/214783}, 2003, pp.~49--57.

\bibitem[HLS24]{HLS21}
Christopher Hacon, Alicia Lamarche, and Karl Schwede, \emph{Global generation of test ideals in mixed characteristic and applications}, Algebr. Geom. \textbf{11} (2024), no.~5, 676--711. \MR{4791070}

\bibitem[HX15]{HaXu15}
Christopher~D. Hacon and Chenyang Xu, \emph{On the three dimensional minimal model program in positive characteristic}, J. Amer. Math. Soc. \textbf{28} (2015), no.~3, 711--744. \MR{3327534}

\bibitem[Kee03]{Keer03}
Dennis~S. Keeler, \emph{Ample filters of invertible sheaves}, J. Algebra \textbf{259} (2003), no.~1, 243--283. \MR{1953719}

\bibitem[Kee08]{Keer08}
\bysame, \emph{Fujita's conjecture and {F}robenius amplitude}, Amer. J. Math. \textbf{130} (2008), no.~5, 1327--1336. \MR{2450210}

\bibitem[KM98]{KoMo98}
J\'anos Koll\'ar and Shigefumi Mori, \emph{Birational geometry of algebraic varieties}, Cambridge Tracts in Mathematics, vol. 134, Cambridge University Press, Cambridge, 1998, With the collaboration of C. H. Clemens and A. Corti, Translated from the 1998 Japanese original. \MR{1658959}

\bibitem[KMM92]{KMM92}
J\'{a}nos Koll\'{a}r, Yoichi Miyaoka, and Shigefumi Mori, \emph{Rational connectedness and boundedness of {F}ano manifolds}, J. Differential Geom. \textbf{36} (1992), no.~3, 765--779. \MR{1189503}

\bibitem[Kol86]{Koll86I}
J\'{a}nos Koll\'{a}r, \emph{Higher direct images of dualizing sheaves. {I}}, Ann. of Math. (2) \textbf{123} (1986), no.~1, 11--42. \MR{825838}

\bibitem[KP17]{KoPa17}
S\'{a}ndor~J. Kov\'{a}cs and Zsolt Patakfalvi, \emph{Projectivity of the moduli space of stable log-varieties and subadditivity of log-{K}odaira dimension}, J. Amer. Math. Soc. \textbf{30} (2017), no.~4, 959--1021. \MR{3671934}

\bibitem[Laz04a]{LazI}
Robert Lazarsfeld, \emph{Positivity in algebraic geometry. {I}}, Ergebnisse der Mathematik und ihrer Grenzgebiete. 3. Folge. A Series of Modern Surveys in Mathematics [Results in Mathematics and Related Areas. 3rd Series. A Series of Modern Surveys in Mathematics], vol.~48, Springer-Verlag, Berlin, 2004, Classical setting: line bundles and linear series. \MR{2095471}

\bibitem[Laz04b]{LazII}
\bysame, \emph{Positivity in algebraic geometry. {II}}, Ergebnisse der Mathematik und ihrer Grenzgebiete. 3. Folge. A Series of Modern Surveys in Mathematics [Results in Mathematics and Related Areas. 3rd Series. A Series of Modern Surveys in Mathematics], vol.~49, Springer-Verlag, Berlin, 2004, Positivity for vector bundles, and multiplier ideals. \MR{2095472}

\bibitem[Liu02]{Liu02}
Qing Liu, \emph{Algebraic geometry and arithmetic curves}, Oxford Graduate Texts in Mathematics, vol.~6, Oxford University Press, Oxford, 2002, Translated from the French by Reinie Ern\'e, Oxford Science Publications. \MR{1917232}

\bibitem[MB81]{More81}
Laurent Moret-Bailly, \emph{Familles de courbes et de vari\'{e}t\'{e}s ab\'{e}liennes sur {${\mathbb P}^1$}. {II}. {E}xemples}, S\'{e}minaire sur les {P}inceaux de {C}ourbes de {G}enre au {M}oins {D}eux, no.~86, Soci\'{e}t\'{e} Math\'{e}matique de France, Paris, 1981, Seminar on Pencils of Curves of Genus at Least Two, pp.~125--140. \MR{3618576}

\bibitem[MS21]{MaSc21}
Linquan Ma and Karl Schwede, \emph{Singularities in mixed characteristic via perfectoid big {C}ohen-{M}acaulay algebras}, Duke Math. J. \textbf{170} (2021), no.~13, 2815--2890. \MR{4312190}

\bibitem[Pat14]{Pat14}
Zsolt Patakfalvi, \emph{Semi-positivity in positive characteristics}, Ann. Sci. \'{E}c. Norm. Sup\'{e}r. (4) \textbf{47} (2014), no.~5, 991--1025. \MR{3294622}

\bibitem[Poo04]{Poon04}
Bjorn Poonen, \emph{Bertini theorems over finite fields}, Ann. of Math. (2) \textbf{160} (2004), no.~3, 1099--1127. \MR{2144974}

\bibitem[PS14]{PoSc14}
Mihnea Popa and Christian Schnell, \emph{On direct images of pluricanonical bundles}, Algebra Number Theory \textbf{8} (2014), no.~9, 2273--2295. \MR{3294390}

\bibitem[Smi97]{Smit97}
Karen~E. Smith, \emph{Fujita's freeness conjecture in terms of local cohomology}, J. Algebraic Geom. \textbf{6} (1997), no.~3, 417--429. \MR{1487221}

\bibitem[ST14]{ScTu14}
Karl Schwede and Kevin Tucker, \emph{Test ideals of non-principal ideals: computations, jumping numbers, alterations and division theorems}, J. Math. Pures Appl. (9) \textbf{102} (2014), no.~5, 891--929. \MR{3271293}

\bibitem[{Sta}]{Stacks}
The {Stacks project authors}, \emph{The {S}tacks project}, \url{https://stacks.math.columbia.edu}.

\bibitem[SZ20]{ShZh20}
Junchao Shentu and Yongming Zhang, \emph{On the simultaneous generation of jets of the adjoint bundles}, J. Algebra \textbf{555} (2020), 52--68. \MR{4081496}

\bibitem[TY23]{TaYo23}
Teppei Takamatsu and Shou Yoshikawa, \emph{Minimal model program for semi-stable threefolds in mixed characteristic}, J. Algebraic Geom. \textbf{32} (2023), no.~3, 429--476. \MR{4622257}

\bibitem[Vie83]{Vieh83}
Eckart Viehweg, \emph{Weak positivity and the additivity of the {K}odaira dimension for certain fibre spaces}, Algebraic varieties and analytic varieties ({T}okyo, 1981), Adv. Stud. Pure Math., vol.~1, North-Holland, Amsterdam, 1983, pp.~329--353. \MR{715656}

\end{thebibliography}
